\documentclass[preprint]{siamltex}
\usepackage{amsmath} 
\usepackage{amssymb}  
\usepackage{graphics}
\usepackage{epsfig}
\usepackage[ruled,vlined]{algorithm2e}
\usepackage{color}

\newtheorem{remark}{Remark}[section]
\newtheorem{assumption}{Assumption}[section]
\newcommand{\beq}{\begin{equation}}
\newcommand{\eeq}{\end{equation}}

\newcommand{\C}{\mathbb C}

\newcommand{\M}{{\mathcal M}}

\newcommand{\R}{\mathbb R}

\def\iu{{\rm i}}

\def\eps{\varepsilon}

\def\sym{\,\mathrm{sym}}

\newcommand{\oeps}{\widehat\eps}

\newcommand{\aleps}{\alpha_\varepsilon}

\renewcommand{\Re}{{\mbox{\rm Re\,}}}

\newcommand{\conj}[1]{\overline{#1}}

\newcommand{\bng}{\color{black}}
\newcommand{\eng}{\color{black}}
\newcommand{\bcl}{\color{black}}
\newcommand{\ecl}{\color{black}}
\title{Finding the nearest passive or non-passive system via Hamiltonian eigenvalue optimization}
\author{---}
\author{Antonio Fazzi\footnotemark[1] \and Nicola Guglielmi\footnotemark[2] \and Christian Lubich\footnotemark[3]}

\begin{document}

\maketitle

\renewcommand{\thefootnote}{\fnsymbol{footnote}}
\footnotetext[1]{Department ELEC, Vrije Universiteit Brussel (VUB), Pleinlaan      2, 1050 Brussels, Belgium. Email: {\tt Antonio.Fazzi@vub.be}}
\footnotetext[2]{Gran Sasso Science Institute, via Crispi 7, L' Aquila,  Italy. Email: {\tt nicola.guglielmi@gssi.it}}
\footnotetext[3]{Mathematisches Institut,
       Universit\"at T\"ubingen,
       Auf der Morgenstelle 10,
       D--72076 T\"ubingen,
       Germany. Email: {\tt lubich@na.uni-tuebingen.de}}
\renewcommand{\thefootnote}{\arabic{footnote}}

\begin{abstract} We propose and study an algorithm for computing a nearest passive system to a given non-passive linear time-invariant system (with much freedom in the choice of the metric defining `nearest', which may be restricted to structured perturbations), and also a closely related algorithm for computing the structured distance of a given passive system to non-passivity. Both problems are addressed by solving eigenvalue optimization problems for Hamiltonian matrices that are constructed from perturbed system matrices. The proposed algorithms are two-level methods that optimize the Hamiltonian eigenvalue of smallest positive real part over perturbations of a fixed size in the inner iteration, using a constrained gradient flow. They optimize over the perturbation size in the outer iteration, which is shown to converge quadratically in the typical case of a defective coalescence of simple eigenvalues approaching the imaginary axis. 
For large systems, we propose a variant of the algorithm that takes advantage of the inherent low-rank structure of the problem. Numerical experiments illustrate the behavior of the proposed algorithms.
\end{abstract}

\begin{keywords} passive control system, structured passivity enforcement, real passivity radius,  Hamiltonian matrix, matrix nearness problem, eigenvalue optimization
\end{keywords}

\begin{AMS}65K05, 93B40, 93B60, 93C05, 15A18 \end{AMS}

\pagestyle{myheadings}
\thispagestyle{plain}
\markboth{A.~Fazzi, N.~Guglielmi and C.~Lubich}{Passivation and distance to non-passivity}

%
%
%

\section{Introduction}
\label{sec:intro}
\subsection{Preliminaries}


We consider 
 linear time-invariant dynamical systems expressed by their state space representation:
\begin{equation}
\label{eq:controlsystem}
\begin{aligned}
\dot{x}(t) &= Ax(t) + Bu(t) \\
y(t) &= Cx(t) + Du(t),
\end{aligned}
\end{equation}
where $A \in \R^{n \times n}, B \in \R^{n \times m}, C \in \R^{p \times n}, D \in \R^{p \times m}$ and we take zero initial values. As for terminology, $x(t)\in \R^n$ is the state vector, while $u(t)\in \R^m$ and $y(t)\in \R^p$ are the input and the output vector, respectively. A system of the form \eqref{eq:controlsystem} is identified by the block matrix 
\begin{equation} \label{X}
X = \begin{pmatrix}
A & B \\ C & D
\end{pmatrix} \in \R^{(n+p)\times(n+m)}
\end{equation}
that collects all the parameters chosen for its representation.
%
\bcl
The following notions of  dissipativity or passivity\footnote{The nomenclature varies among different scientific and engineering communities.} will be central to this paper.
A system \eqref{eq:controlsystem} that is quadratic, i.e., $p=m$, is called {\it passive} (or {\it passive in the immittance representation}) if every input $u\in L^2(\R_+,\R^m)$ and its corresponding output $y$ satisfy the relation
\begin{equation}\label{passive-immittance}
\int_0^\infty y(t)^\top u(t)\, dt \ge 0.
\end{equation}
A system \eqref{eq:controlsystem} is called {\it contractive} (or {\it passive in the scattering representation}) if every input $u\in L^2(\R_+,\R^m)$ bounds the corresponding output $y$ by
\begin{equation}\label{passive-scattering}
\int_0^{\infty} \|y(t) \|^2 \  dt \leq \int_0^{\infty} \| u(t) \|^2 \,  dt,
\end{equation}
where the norm $\|\cdot\|$ is the Euclidean norm. 

\subsection{Objective}
\label{subsec:objective}
%
In this paper we consider algorithms for the following two problems, which we first state in an idealized version that will be reformulated later. Here `passive' refers to either \eqref{passive-immittance} or \eqref{passive-scattering} and the notion `nearest' is yet to be specified:
\begin{itemize}
\item {\it Passivity enforcement.} Given a non-passive system  \eqref{X}, 
compute a nearest perturbed system   $X + \Delta X$   that is passive.
\item {\it Distance to the nearest non-passive system.} Given a passive system \eqref{X},
compute a nearest perturbed system   $X + \Delta X$   that is no more passive.
\end{itemize}
The two problems can be combined in asking for a nearest  passive system whose distance to non-passivity is at least $\delta$.
\ecl

`Nearest' could refer to the distance to $X$ in the Frobenius norm or a weighted Frobenius norm, but different concepts of nearness are often imposed in the engineering literature. One approach, described in \cite{Grivet2004,Grivet2015}, is to keep the matrices $A,D$ and also $B$ fixed and to perturb only the state-output matrix $C$ in such a way that the Frobenius norm of $\Delta C\, Q^T$ is kept minimal, where $Q^T\in \R^{n\times n}$ is a Cholesky factor of the controllability Gramian $G_c=Q^TQ$. This Gramian depends only on the matrices $A$ and $B$ and is positive definite for a controllable system, as will be assumed henceforth when we consider this guiding example. In other cases, nearness may also be understood under constraints of preserving a sparsity pattern or symmetry.

A framework that includes all nearness options in the literature as far as known to us, and which allows for {\it structured passivation}, is to specify a linear map $L:\R^{k\times l} \to \R^{(n+p)\times(n+m)}$ and to consider only perturbations of the form\footnote{In this paper we put arguments of linear maps in square brackets.}
\begin{equation}\label{L}
\Delta X = L[\Delta Z]\ \text{ and minimize } \| \Delta Z \|_F.
\end{equation}
In the example considered above, where only $C$ is perturbed and $\|\Delta C\, Q^T\|_F$ is to be minimized, we have $k=p$ and $l=n$ and
\begin{equation}\label{L-example}
L[\Delta Z] = \begin{pmatrix} 0_{n\times n} & 0_{n\times m} \\ \Delta Z \,Q^{-T} & 0_{p\times m} \end{pmatrix} .
\end{equation}
Another case of interest is sparsity-preserving passivation, where $L[\Delta Z]$ represents a sparse matrix whose nonzero entries are collected in a vector $\Delta Z$. 

\bcl
Passivity enforcement has been addressed in the literature with various approa\-ches. Grivet-Talocia \cite{Grivet2004}  used the characterization of passivity in terms of Hamiltonian matrices and applied first-order perturbation theory to move and coalesce imaginary eigenvalues of the associated Hamiltonian matrix. Schr\"oder \& Stykel \cite{SchS07} took that approach further, using structure-preserving algorithms to compute the required eigenvalues and eigenvectors. Br\"ull \& Schr\"oder \cite{BruS13}  substantially extended the approach to  descriptor systems and to more general notions of dissipativity that they characterized in terms of structured matrix pencils, which they used together with first-order perturbation theory to move imaginary eigenvalues. Gillis \& Sharma \cite{GilS18} studied the problem of finding the nearest port-Hamiltonian descriptor system. Such systems are always passive and conversely, with the appropriate terminology, every extended strictly passive system is port-Hamiltonian. For related work on port-Hamiltonian systems see also \cite{GilS17,MehMS16,MehMS17}.

The problem of determining the distance to non-passivity, or passivity radius, was studied by Overton \& Van Dooren \cite{OveVD05} for complex perturbations of the system matrices, but the case of real perturbations as considered here has been left open.

\subsection{Recap: Hamiltonian matrices and pencils related to passivity}
\label{subsec:intro-ham}
By the Plan\-cherel theorem, the passivity notions \eqref{passive-immittance} and \eqref{passive-scattering} can be characterized in terms of the transfer function $H(s) = C(s I - A)^{-1} B + D$,
which  describes the relation between the Laplace transform of the  input and the Laplace transform of the output through the expression $\mathcal{L}y(s) = H(s) \mathcal{L}u(s)$, for $\Re s \ge 0$. 

Passivity \eqref{passive-immittance} is equivalent to
\begin{equation}\label{pr}
H(s)+H(s)^* \text{ is positive semidefinite for $\Re s \ge 0$}.
\end{equation}

Contractivity \eqref{passive-scattering} is equivalent to
\begin{equation}\label{br}
I-H(s)^*H(s) \text{ is positive semidefinite for $\Re s \ge 0$}.
\end{equation}
In a field rich in terminology, the properties \eqref{pr} and \eqref{br} of the transfer function are known as {\it positive realness} and {\it bounded realness}, respectively. In the following it will always be assumed that
$$
\text{the eigenvalues of the state space matrix $A$ have negative real part.}
$$ 
It follows that the poles of the transfer function $H(s)$ have negative real part, and by the maximum principle, it then suffices that conditions \eqref{pr} and \eqref{br} be satisfied for $s=\iu\omega$ on the imaginary axis.
 
It is a remarkable fact, going back to Boyd, Balakrishnan \& Kabamba~\cite[Theorem~2]{Boyd1989} and fully presented, e.g., by Grivet-Talocia \& Gustavsen~\cite[Chapter 9]{Grivet2015},
that strict positive realness  and strict bounded realness  (i.e., \eqref{pr} and \eqref{br} with `positive definite' instead of `positive semidefinite') are characterized by the location of the
eigenvalues of an extended Hamiltonian matrix pencil that contains the system matrices $(A,B,C,D)$ and their transposes as blocks. Under a nondegeneracy condition on the matrix $D$, these eigenvalues are  those of a Hamiltonian matrix built from the system matrices. This eigenvalue criterion eliminates any dependence on the complex frequencies~$s$.

We recall that a real Hamiltonian matrix  is a matrix $M \in \R^{2n \times 2n}$ that satisfies
\begin{equation}
 (JM)^{\top} = JM, \ \ \ \text{with } \ J=\begin{pmatrix}
0 & I_n \\ -I_n & 0
\end{pmatrix}.
\end{equation}
The spectrum of a real Hamiltonian matrix is known to be symmetric with respect to both the real and the imaginary axis.

Strict positive realness  is characterized by the condition that {\it the extended Hamiltonian matrix pencil
\begin{equation}\label{mp-p}
\lambda
\begin{pmatrix}
I_n & 0 & 0 \\
0 & I_n & 0 \\
0 & 0 & 0 
\end{pmatrix}
-
\begin{pmatrix}
A & 0 & B \\
0 & -A^\top & -C^\top \\
C & B^\top & D+D^\top
\end{pmatrix}
\end{equation}
has no eigenvalues on the imaginary axis}. If $D+D^\top$ is positive definite, then this matrix pencil has the same eigenvalues as
the {\it Hamiltonian matrix}
\begin{equation}
\label{M-p}
M_p(X) =  \begin{pmatrix}
A & 0\\ 0 & -A^T
\end{pmatrix}
-
 \begin{pmatrix}
B \\-C^\top
\end{pmatrix}
(D+D^\top)^{-1}
 \begin{pmatrix}
C^\top \\ B
\end{pmatrix}^\top\ 
 \in \R^{2n\times 2n}.
\end{equation}

Strict bounded realness  is characterized by the condition that {\it the extended Hamiltonian matrix pencil
\begin{equation}\label{mp-b}
\lambda
\begin{pmatrix}
I_n & 0 & 0 & 0\\
0 & I_n & 0 &0\\
0 & 0 & 0 & 0 \\
0 & 0 & 0 & 0 
\end{pmatrix}
-
\begin{pmatrix}
A & 0 & B & 0\\
0 & -A^\top &0 & -C^\top \\
0 & B^\top & - I_m & D^\top \\
C & 0 & D & I_p
\end{pmatrix}
\end{equation}
has no eigenvalues on the imaginary axis}. 
If $I-D^\top D$ is positive definite (i.e., ${\| D\|_2 <1}$), then this matrix pencil has the same eigenvalues as
the {\it Hamiltonian matrix} 
\begin{equation}
\label{M-b}
M_b(X) = \begin{pmatrix}
A & 0\\ -C^\top C & -A^T
\end{pmatrix}
+
 \begin{pmatrix}
B \\-C^\top D 
\end{pmatrix}
(I-D^\top \! D)^{-1}
 \begin{pmatrix}
C^\top D \\ B
\end{pmatrix}^\top\ 
 \in \R^{2n\times 2n}.
\end{equation}
Concerning the computation of all or a selection of these eigenvalues, it seems to depend on the situation whether it is preferable to use an eigenvalue algorithm for Hamiltonian matrices or for extended Hamiltonian matrix pencils, as implemented in the SLICOT library (\cite{slicot}, {\tt http://slicot.org/}) for either option. For ill-conditioned matrices $D+D^\top$ or
$I-D^\top D$, working with the pencil appears preferable.
In any case, using a structure-preserving eigenvalue solver is often distinctly favourable over using a standard general eigenvalue solver, especially for eigenvalues on and close to the imaginary axis as are of interest  here; see, e.g., \cite{BenLMV15} and references therein.

For large sparse systems for which shifted linear systems of equations with $A$ and $A^\top$ can be solved efficiently and for which the states have much higher dimension than the inputs and outputs,  the fact that the above Hamiltonian matrices are low-rank perturbations to the block-diagonal matrix with blocks $A$ and $-A^\top$, is beneficial for the computation of selected eigenvalues; see \cite[Section~9.4.2]{Grivet2015} and Section~\ref{sec:low-rank} below.

\subsection{Outline of the approach of this paper}
We address the two problems stated in Subsection~\ref{subsec:objective} in the following (closely related but not fully equivalent) reformulation as a Hamiltonian eigenvalue optimization problem, where $M(X)$ is either $M_p(X)$ or $M_b(X)$ and the metric defining `nearest' is chosen in the general form \eqref{L}, with \eqref{L-example} as our guiding example.
\ecl
\begin{enumerate}
\item {\it (Passivity enforcement).} Given a system  \eqref{X} 
for which the Hamil\-tonian matrix $M(X)$ has some purely imaginary eigenvalues, 
and given $\delta>0$, compute a nearest perturbed system   $X + \Delta X$ such that all eigenvalues of $M(X + \Delta X)$ have a real part of absolute value not smaller than~$\delta$.
\item {\it (Distance to the nearest non-passive system).} Given a  system \eqref{X} 
for which the Hamiltonian matrix $M(X)$ has no purely imaginary eigenvalues, 
compute a nearest perturbed system   $X + \Delta X$   such that the Hamiltonian matrix $M(X + \Delta X)$ has at least one pair of purely imaginary eigenvalues.
\end{enumerate}
Our algorithms for Problems 1 and 2 can be combined in searching for a nearest passive system whose distance to non-passivity is at least $\delta$.

\bcl
In both Problems 1 and 2, the perturbations of the system matrix need to be constrained such that the asymptotic stability of the perturbed state-space matrix $A$ is preserved, and there are also positivity or boundedness constraints on the perturbed matrix $D$. This is trivially satisfied if $A$ and $D$ are not subjected to perturbations. Strong arguments for this case are given in \cite[Section~10.2.5]{Grivet2015}. 
\ecl

Our proposed algorithms are {\it two-level iterative methods} similar to \cite{Guglielmi2015,Guglielmi2017}. In the {\it inner iteration}, for a given $\eps>0$, we use a gradient ascent / descent algorithm to solve the eigenvalue optimization problem, over $E\in \R^{k\times l}$ with $\| E \|_F =1$, 
\begin{equation}\label{eig-opt}
E_\eps = \begin{matrix} \arg \max \\ \arg \min \end{matrix} \ 
\phi_\eps(E)  \quad\text{ with }\quad \phi_\eps(E)= \Re \lambda(M(X+ L[\eps E])),
\end{equation}
where the maximum is taken for Problem 1 and the minimum for Problem 2. Here,
$\lambda(M)$ is an eigenvalue of minimal nonnegative real part (chosen with the largest, nonnegative imaginary part) of a real Hamiltonian matrix $M$.
\bcl
For both Problems~1 and~2, the optimization is done under the (inequality) constraint that the eigenvalues of the perturbed state space matrix have a real part smaller than some given negative threshold.
\ecl
The {\it outer iteration} determines the  smallest $\eps$ such that
$\phi_\eps(E_\eps)= \delta$ for a given small threshold $\delta>0$. \bcl This uses a mixed Newton/bisection method and, for very small $\delta$, the asymptotic square-root behavior $ \phi_\eps(E_\eps)\sim \sqrt{\eps-\oeps}$ as the eigenvalue tends to the imaginary axis
at perturbation size $\oeps$. \ecl

 \bcl
In contrast to existing passivity enforcement algorithms, we do not move Hamiltonian eigenvalues on the imaginary axis. Instead, we move eigenvalues with smallest positive real part toward the imaginary axis, starting from a non-optimal passive perturbation $X+L[\Delta Z_0]$ of the original, non-passive system. The matrix $\Delta Z_0=\eps_0 E_0$ provides an initial iterate for an iteration over passive matrices that tend to $X+L[\widehat{\Delta Z}]$ with $\widehat{\Delta Z}=\oeps E(\oeps)$ of (locally) minimal Frobenius norm, gradually reducing the perturbation size from $\eps_0$ down to the optimal $\oeps$.
The algorithm can thus be used to improve the results of some existing computationally inexpensive but non-optimal passivity-enforcing algorithm by taking its result as starting value.
\ecl

Both Problems 1 and 2 are nonconvex and nonsmooth optimization problems. Our algorithm is not guaranteed to find the global optimum but it computes a passive perturbation of the original system for Problem 1 and a non-passive perturbation for Problem 2 that are locally optimal. Running the algorithm with several different starting values reduces the potential risk of getting stuck in a local optimum.

\ecl
The related problem of finding a nearest Hamiltonian matrix with / without eigenvalues on the imaginary axis has previously been addressed in
\cite{Alam2011, Guglielmi2015}. However, there the resulting perturbed Hamiltonian matrix  is not, in general, a matrix $M(X+\Delta X)$ associated with some linear control system $X+\Delta X$. In contrast, in this paper we put a structured perturbation $\Delta X= L[\Delta Z]$ directly on $X$ and not on the Hamiltonian matrix, and we minimize the Frobenius norm of $\Delta Z$  as opposed to the Frobenius-norm distance between the perturbed and unperturbed associated Hamiltonian matrices.

\subsection{Contents}
\bcl
In Section~\ref{sec:E} we study the optimization of  the eigenvalue of smallest positive real part of Hamiltonian matrices corresponding to structured perturbations $L[\Delta Z]$ of the system matrix, with $\Delta Z$ of a fixed norm~$\eps$. We derive a constrained gradient system and give a characterization of its stationary points.

In Section~\ref{sec:eps} we study, for Problem 1 (passivity enforcement), the equation for determining the optimal perturbation size $\eps$, which is near a point of coalescence on the imaginary axis of a pair of eigenvalues off the imaginary axis.  Under appropriate assumptions, we prove the asymptotic square-root behavior $ \phi_\eps(E_\eps)\sim \sqrt{\eps-\oeps}$ for the real part of the optimal eigenvalue corresponding to perturbation size $\eps$ close to coalescence at $\oeps$.\ecl

In Section~\ref{sec:alg1} we describe the two-level algorithm for Problem~1. 
This algorithm determines an optimal passivity-enforcing perturbation $\Delta X=L[\Delta Z]$ with $\Delta Z=\eps E$ and $E$ of unit Frobenius norm, where in the inner iteration an optimal matrix $E(\eps)$ is computed for a fixed perturbation size $\eps$, following the constrained gradient flow of Section~\ref{sec:E}, and the optimal perturbation size $\eps$ is computed in the outer iteration \bcl using an approach that uses a mixed Newton/bisection method or the asymptotic square-root behavior with bisection as a resort. \ecl
There is a completely analogous algorithm for Problem 2 (distance to passivity), which is briefly addressed in Section~\ref{sec:alg2}.

In Section~\ref{sec:low-rank} we show how the inherent low-rank structure of the problem can be exploited for large sparse systems.
\bcl
The gradient system of Section~\ref{sec:E}, whose stationary points are of low rank, is restricted to matrices of the appropriate low rank. This restricted gradient system still has exactly the same stationary points as the original gradient system. Using the low-rank structure is beneficial both in the eigenvalue optimization and the eigenvalue computations.
\ecl

In Section~\ref{sec:num} we present a few illustrative numerical experiments.


\section{Optimizing the eigenvalue of smallest positive real part over perturbations of a fixed norm}
\label{sec:E}

\subsection{Recap: derivatives of simple eigenvalues}
We repeatedly use the following basic result.
\begin{lemma}\label{lem:basic} \cite{Kato1976}
Consider a differentiable matrix valued function $C(t)$ for $t$ in a neighborhood of $t_0\in\R$. Let $\lambda(t)$ be a continuous path of simple eigenvalues of $C(t)$ for $t$ near $t_0$. Let $x_0, y_0$ be the left and right eigenvectors, respectively, of $C(t_0)$ corresponding to $\lambda(t_0)$. Then $x_0^*y_0 \neq 0$ and $\lambda(t)$ is differentiable at~$t_0$ with
\begin{equation}
\dot{\lambda}(t_0) = \frac{x_0^* \dot{C}(t_0) y_0}{x_0^* y_0}\,.
\end{equation}
\end{lemma}
Since we have  $x_0^*y_0 \neq 0$, we will always apply the  normalization
\begin{equation}\label{eq:scalyx}
\| x_0 \|=1, \ \ \|y_0 \| =1, \ \ x_0^* y_0 \text{ is real and positive.}
\end{equation}

\subsection{Free gradient of the eigenvalue optimization problem}
Given $\eps>0$ and a  smooth path of matrices $E(t) \in \R^{k\times l}$ of unit Frobenius norm, at $X(t) = X + L[\eps E(t)]$ we  compute the change of $\phi_\eps(E(t))$ of \eqref{eig-opt} as $t$ varies. We use Lemma~\ref{lem:basic}  for the eigenvalue $\lambda(t)=  \lambda\bigl(M(X+L[\eps E(t)])\bigr)$ of minimal positive real part (with nonnegative imaginary part), which we assume to be simple, and denote by $x(t),y(t)$ the corresponding left and right eigenvectors in the above normalization and we set 
\begin{equation}\label{kappa}
\kappa(t) :=\frac 1{x(t)^* y(t)}>0.
\end{equation}
 We  obtain from Lemma~\ref{lem:basic}
\begin{align*}
\frac{d}{dt} \,\phi_\eps(E(t)) &= \frac {d}{dt} \,\Re \lambda\bigl(M(X+L[\eps E(t)])\bigr) 
\\
&= \kappa(t) \,\Re \Bigl(x(t)^* \frac{d}{dt} M(X+L[\eps E(t)]) y(t)\Bigr).
\end{align*}
We denote by $\langle U,V\rangle = \text{tr}(U^\top V)$ the Frobenius inner product on $\R^{(n+p)\times (n+m)}$ or $\R^{2n\times 2n}$ or $\R^{k\times l}$, as will be clear from the context. 
We then rewrite the expression above as
\begin{align*}
\frac{d}{dt} \phi_\eps(E(t)) &=  \kappa(t) \Bigl\langle M'(X+L[\eps E(t)])[L[\eps \dot E(t)]], \Re x(t)y^*(t)  \Bigr\rangle
\\
&=  \kappa(t) \Bigl\langle L[\eps \dot E(t)], M'(X+L[\eps E(t)])^\ast[\Re x(t)y^*(t)]  \Bigr\rangle
\\
&=  \kappa(t)\eps \Bigl\langle  \dot E(t),  L^\ast \bigl[ M'(X+L[\eps E(t)])^\ast[\Re x(t)y^*(t)]\bigr]  \Bigr\rangle,
\end{align*}
where $L^\ast:\R^{(n+p)\times(n+m)}\to\R^{k\times l}$ is the adjoint of the linear map $L:\R^{k\times l} \to \R^{(n+p)\times(n+m)}$, i.e.,
$$
\langle L[U],V \rangle = \langle U, L^\ast [V] \rangle \qquad\text{for all } \ U\in \R^{k\times l}, V \in \R^{(n+p)\times(n+m)},
$$
and  $M'(X)^\ast: \R^{2n\times 2n} \to \R^{(n+p)\times(n+m)}$ is the adjoint of the linear map $M'(X): \R^{(n+p)\times(n+m)} \to \R^{2n\times 2n}$. Summarizing this calculation, we have the following result.

\begin{lemma} \label{lem:gradient}
With the above notation, if $\lambda(t)$ is a simple eigenvalue, then
$$
\frac1{\eps \kappa(t)} \,\frac{d}{dt} \phi_\eps(E(t)) = \bigl\langle  G_\eps(E(t)),  \dot E(t) \bigr\rangle
$$
with the (rescaled) gradient
\begin{equation}
G_\eps(E)= L^\ast M'(X+L[\eps E])^\ast[\Re(x y^*)] \in \R^{k\times l}.
\label{eq:grad}
\end{equation}
Here, $x,y \in \C^{2n}$ are the left and right eigenvectors, respectively, corresponding to the simple eigenvalue $\lambda$ (of minimal positive real part) of the Hamiltonian matrix $M(X+L[\eps E])$. 
\end{lemma}


%

\bcl
In the following two lemmas we give explicit expressions for $M'(X)^\ast [W]$ for a general matrix $W$ of the same dimension $2n\times 2n$ as the Hamiltonian matrix $M(X)$. 
We first give the result for $M=M_p$ of \eqref{M-p} and then for $M=M_b$ of \eqref{M-b}.
Here, $\sym(Z)=\tfrac12(Z+Z^\top)$ denotes the symmetric part of a matrix $Z$.

\begin{lemma}
\label{lemma:formulaoperator-p}
For $W=\begin{pmatrix}
W_{11} & W_{12} \\ W_{21} & W_{22} 
\end{pmatrix}\in \R^{2n\times 2n}$ partitioned according to the $n\times n$ blocks of~$M_p(X)$, we have $M_p'(X)^\ast[W] = V= \begin{pmatrix}
V_A & V_B \\ V_C & V_D
\end{pmatrix}\in \R^{(n+m)\times(n+m)}$ partitioned according to the blocks of the system matrix~$X=\begin{pmatrix}
A & B \\ C & D
\end{pmatrix}$, where (with $T=D+D^\top$ for short)
\begin{equation}
\label{eq:Ei-p}
\begin{aligned}
V_A &= W_{11} - W_{22}^\top  \\
V_B &=-(W_{11} - W_{22}^\top) C^\top  T^{-1} - 2 \sym(W_{12}) B T^{-1} \\
V_C &= -T^{-1} B^\top (W_{11} - W_{22}^\top) +  T^{-1} C\,2\sym(W_{21}) \\
V_D&= 2\,\sym\left(T^{-1} (B^\top,-C) \,W \begin{pmatrix} C^\top \\ B \end{pmatrix} T^{-1} \right).
\end{aligned}
\end{equation}
\end{lemma}

\begin{proof} For any path $X(t)=\begin{pmatrix}
A(t) & B(t) \\ C(t) & D(t)
\end{pmatrix}$ we have $\dot{M_p} = \frac d{dt} M_p(X) = M_p'(X)[\dot{X}]$ and hence
$$
\langle \dot{X},  M_p'(X)^\ast [W] \rangle = \langle M_p'(X) [\dot{X}], W \rangle = \langle \dot{M_p}, W \rangle.
$$
Differentiation of \eqref{M-p} yields
\begin{align*}
\dot M_p  =&\ \begin{pmatrix}
\dot A & 0\\ 0 & -\dot A^\top
\end{pmatrix}
-
 \begin{pmatrix}
\dot B \\-\dot C^\top
\end{pmatrix}
T^{-1}
 \begin{pmatrix}
C^\top \\ B
\end{pmatrix}^\top
\\
&\
+ 
 \begin{pmatrix}
B \\-C^\top
\end{pmatrix}
T^{-1} \dot T  T^{-1}
 \begin{pmatrix}
C^\top \\ B
\end{pmatrix}^\top
-
 \begin{pmatrix}
B \\-C^\top
\end{pmatrix}
T^{-1}
 \begin{pmatrix}
\dot C^\top \\ \dot B
\end{pmatrix}^\top.
\end{align*}
We note that
\begin{align*}
&\left\langle  \begin{pmatrix}
\dot A & 0\\ 0 & -\dot A^\top
\end{pmatrix} , W \right\rangle= \langle \dot A, W_{11}-W_{22}^\top \rangle
\end{align*}
\vspace{-4mm}
\begin{align*}
&\left\langle
\begin{pmatrix}
\dot B \\-\dot C^\top
\end{pmatrix}
T^{-1}
 \begin{pmatrix}
C^\top \\ B
\end{pmatrix}^\top , W  \right\rangle=  \left\langle
 \begin{pmatrix}
\dot B \\-\dot C^\top 
\end{pmatrix} , W  \begin{pmatrix}
C^\top \\ B
\end{pmatrix} T^{-1} \right\rangle 
\\
&\qquad = \langle  \dot B, (W_{11} C^\top + W_{12} B) T^{-1} \rangle -
\langle \dot C,  ( (W_{21} C^\top + W_{22} B) T^{-1})^\top \rangle
\end{align*}
\vspace{-4mm}
\begin{align*}
& \left\langle
\begin{pmatrix}
B \\-C^\top
\end{pmatrix}
T^{-1} \dot T  T^{-1}
 \begin{pmatrix}
C^\top \\ B
\end{pmatrix}^\top , W \right\rangle = 2 \left\langle \dot D, \sym\left(T^{-1} (B^\top,-C) \,W \begin{pmatrix} C^\top \\ B \end{pmatrix} T^{-1} \right) \right\rangle
\end{align*}
\vspace{-4mm}
\begin{align*}
&  \left\langle
\begin{pmatrix}
B \\-C^\top
\end{pmatrix}
T^{-1}
 \begin{pmatrix}
\dot C^\top \\ \dot B
\end{pmatrix}^\top, W \right\rangle = \langle (\dot C, \dot B^\top), T^{-1} (B^\top,-C)W \rangle
\\
&\qquad = \langle \dot C, T^{-1} (B^\top W_{11} - C W_{21}) \rangle +
\langle \dot B, (T^{-1} (B^\top W_{12} - C W_{22})^\top \rangle
\end{align*}
so that finally we obtain, with $V$ as defined in the lemma, 
$$ 
\langle \dot X, M_p'(X)^\ast [W] \rangle =
\langle \dot{M_p}, W \rangle  = \langle\dot A,V_A \rangle + \langle\dot B,V_B \rangle + \langle\dot C,V_C \rangle +
\langle\dot D,V_D \rangle = \langle \dot X, V \rangle
$$
for all $\dot X=\begin{pmatrix}
\dot A & \dot B \\ \dot C & \dot D
\end{pmatrix}\in \R^{(n+m)\times(n+m)}$. This yields $M_p'(X)^\ast [W]=V$, which is the stated result.
\end{proof}

\medskip
By a similar calculation we also obtain the following result.
\ecl

\begin{lemma}
\label{lemma:formulaoperator-b}
For $W=\begin{pmatrix}
W_{11} & W_{12} \\ W_{21} & W_{22} 
\end{pmatrix}\in \R^{2n\times 2n}$ partitioned according to the $n\times n$ blocks of~$M_b(X)$, we have $M_b'(X)^\ast[W] = \begin{pmatrix}
V_A & V_B \\ V_C & V_D
\end{pmatrix}\in \R^{(n+p)\times(n+m)}$ partitioned according to the blocks of the system matrix~$X=\begin{pmatrix}
A & B \\ C & D
\end{pmatrix}$, where (with $R=D^\top D - I$ and $S=DD^\top-I$)
\begin{equation}
\label{eq:Ei-b}
\begin{aligned}
V_A &= W_{11} - W_{22}^\top  \\
V_B &=-(W_{11} - W_{22}^\top) C^\top D R^{-1} - 2 \sym(W_{12}) B R^{-1} \\
V_C &=- D R^{-1} B^\top (W_{11} - W_{22}^\top) + 2 S^{-1} C\sym(W_{21}) \\
V_D&=-C (W_{11} - W_{22}^\top)^\top B R^{-1} + 2D \sym(R^{-1} B^\top (W_{11} - W_{22}^\top) C^\top D R^{-1})  \\
  &+ 2DR^{-1}B^\top \sym (W_{12}) B R^{-1} - 2 S^{-1} C \sym (W_{21}) C^\top S^{-1} D.
\end{aligned}
\end{equation}
\end{lemma}

We further note that for the particular choice \eqref{L-example} of $L$, we have 
\begin{equation}\label{L-star}
L^\ast[V]=V_C\,Q^{-1} \quad\text{ for } \quad V=\begin{pmatrix}
V_A & V_B \\ V_C & V_D
\end{pmatrix}
\in \R^{(n+p)\times (n+m)}
\end{equation}
partitioned according to the blocks of~$X$.

\bcl
If $L[\Delta Z]$ represents a sparse matrix whose nonzero entries are collected in a vector $\Delta Z$, then $L^\ast[V]$ is simply the vector that retains those entries of $V$ that correspond to the given sparsity pattern.
\ecl

\subsection{Norm-constrained gradient flow}
A differential equation for $E(t)$ of unit Frobenius norm along which $\phi_\eps(E(t))$ increases (as desired for Problem~1: passivation) or decreases (as desired for Problem~2: computing the distance to the nearest non-passive system), is given by a {\it gradient flow constrained to the unit sphere} $\| E \|_F =1$ in the direction of steepest ascent (for Problem 1) or steepest descent (for Problem 2) with the $+$ or $-$ sign, respectively:
\begin{equation}
\label{ode-E}
\dot{E} = \pm \bigl( G_\eps(E)  - \mu E \bigr),
\end{equation}
where the Lagrange multiplier $\mu = \langle G_\eps(E), E \rangle$ is chosen to ensure that $\frac {d}{dt} \| E \|_F^2 = 2\langle E, \dot{E} \rangle $ is identically zero when  $\| E \|_F=1$.


This differential equation for $E$ is solved numerically into a stationary point. We note the following properties (cf.~\cite[Theorem~3.1]{Guglielmi2017}).

\begin{theorem}\label{thm:stat}
Along solutions $E(t)$ of  {\rm (\ref{ode-E})} of unit Frobenius norm for which the eigenvalue $\lambda(t)$ of $M(X+L[\eps E(t)])$ 
with smallest positive real part is simple \bcl for almost all $t$, we have at these $t$\ecl
$$
\pm \frac{d}{dt} \phi_\eps\bigl(E(t)\bigr) \ge 0,
$$
and the following statements are equivalent if $G_\eps(E)\ne 0$:
\begin{enumerate}
\item $\frac{d}{dt} \phi_\eps\bigl(E(t)\bigr) = 0$.
\item $\dot E = 0$.
\item $E$ is a real multiple of $G_\eps(E)$.
\end{enumerate}
\end{theorem}

\begin{proof} The short proof is essentially the same as that of Theorem~3.1 in~\cite{Guglielmi2017}. We include it for the convenience of the reader.
Let $G=G_\eps(E)$ for short. Since $\mu=\langle G,E \rangle$ and $\|E\|_F=1$,
{\it 3.} implies {\it 2.}, and clearly {\it 2.} implies {\it 1.} So it remains to show that {\it 1.} implies {\it 3.} We note that 
\begin{align*}
&\pm \frac1{\eps \kappa(t)} \,\frac{d}{dt} \phi_\eps\bigl(E(t)\bigr) =\pm \langle G,\dot E \rangle =\langle G, G-\mu E \rangle 
= \|G\|_F^2 - \langle G,E \rangle^2 \ge 0.
\end{align*}
The last inequality holds by the Cauchy-Schwarz inequality and $\|E\|_F=1$. This inequality is strict unless $G$ is a real multiple of $E$. Finally, $G$ is nonzero by assumption. Hence, {\it 1.} implies {\it 3.}
\end{proof}

\begin{remark} Along a trajectory $E(t)$, the eigenvalue with smallest positive real part $\lambda(t)=\lambda(M(X+L[\eps E(t)]))$ may become discontinuous (because a different branch of eigenvalues gets to have smallest positive real part) or become a multiple eigenvalue at some instance $t$. In the case of a discontinuity, the differential equation is further solved, with an ascent/descent of the eigenvalue with smallest positive real part until finally a stationary point is approximately reached. \bcl Eigenvalues of  Hamiltonian matrices coalesce with their mirrored eigenvalue when they approach the real or imaginary axis but along a path of such eigenvalues, which does not become stationary on a point of coalescence, the eigenvalues will typically split again immediately after the coalescence.  Even if some continuous trajectory runs into a coalescence, \ecl this is highly unlikely to happen after discretization of the differential equation, 
and so the computation will not be affected.
\end{remark}

\newcommand{\Ram}{{\rm R}}
\newcommand{\xA}{x_A}
\newcommand{\xD}{x_D}
\newcommand{\yA}{y_A}
\newcommand{\yD}{y_D}
\bng

\subsection{Constraints on the perturbed state space matrix and the perturbed feedthrough matrix}

If the state space matrix $A$ is allowed to be perturbed (which is not the case with \eqref{L-example}), then we must additionally 
ensure that all eigenvalues of the perturbed matrix still have negative real part, say smaller than some given 
threshold $-\vartheta_A < 0$. 

Moreover, if the feedthrough matrix $D$ is perturbed, depending on the framework (positive realness or bounded realness) 
we have to impose on the perturbed matrix $\widehat D$ that either ${\rm sym}(\widehat D)=\tfrac12(\widehat D + \widehat D^\top)$ or
$I - \widehat D^\top \widehat D$ is positive definite, say with smallest eigenvalue greater than $\vartheta_D>0$. In the following we consider only the first case (positive realness). The case of bounded realness is treated analogously.

We write for $L$ of \eqref{L}
\begin{equation*}
L[\Delta Z] = \left( \begin{array}{cc} L_A[\Delta Z] & L_B[\Delta Z] \\ L_C[\Delta Z] & L_D[\Delta Z] \end{array} \right) 
\end{equation*}
and again decompose $\Delta Z=\eps E$ with $\eps>0$ and $E$ of Frobenius norm 1.
%
%
%

%
%
%
%

We define the penalization functions
\begin{eqnarray*}
\phi_{A,\eps} \left(E \right) & = & \frac12 \Bigl(  \vartheta_A + \Re \lambda_{\rm max} \bigl(A + L_A[\eps E] \bigr) \Bigr)_+^{\ 2}
\\
\phi_{D,\eps} \left(E \right) & = & \frac12 \Bigl(  \vartheta_D - \lambda_{\rm min}\bigl({\rm sym}\left(D + L_D[\eps E] \right) \bigr) \Bigr)_+^{\ 2},
\end{eqnarray*}
where $(x)_+=\max(x,0)$ and $\Re\lambda_{\rm max}(\cdot)$ is the maximal real part of an eigenvalue of the indicated matrix (the spectral abscissa), and $\lambda_{\rm min}(\cdot)$ is the minimal eigenvalue of a symmetric matrix.
In the same way as Lemma~\ref{lem:gradient} we obtain the following result.

\begin{lemma} \label{lem:gradP}
For a path $E(t)$, let $\lambda_{\rm max}(t)$ be a simple eigenvalue of $A + \eps L_A[E(t)]$. Then,
\[
\frac1{\eps \kappa_A(t)} \,\frac{d}{dt} \phi_{A,\eps} \left( L[E(t)] \right) = \bigl\langle  G_{A,\eps}(E(t)),  \dot E(t) \bigr\rangle
\]
with the (rescaled) gradient
\begin{equation}
G_{A,\eps}(E) = L_A^\ast [\Re(\xA \yA^*)] \Bigl(  \vartheta_A + \Re \lambda_{\rm max} \bigl(A + \eps L_A[E] \bigr) \Bigr)_+ \in \R^{k\times l},
\label{eq:gradPA}
\end{equation}
where $\xA, \yA$ are normalized left and right eigenvectors  to $\lambda_{\rm max}(A + \eps L_A[\eps E]))$, and $\kappa_A=1/(\xA^* \yA)>0$.

Similarly, if $\lambda_{\rm min}(t)$ is a simple eigenvalue of ${\rm sym}\left( D + \eps L_D[E(t)] \right)$, then
\[
\frac1{\eps} \,\frac{d}{dt} \phi_{D,\eps} \left( L[E(t)] \right) = \bigl\langle  G_{D,\eps}(E(t)),  \dot E(t) \bigr\rangle
\]
with the (rescaled) gradient
\begin{equation}
G_{D,\eps}(E) = L_D^\ast [\yD \yD^\top] \Bigl( \vartheta_D - \lambda_{\rm min}\bigl({\rm sym}(D + L_D[\eps E])  \bigr) \Bigr)_+ \in \R^{k\times l},
\label{eq:gradPD}
\end{equation}
where $ \yD$ is the normalized real right eigenvector  to $\lambda_{\rm min}\left({\rm sym} (D + L_D[\eps E]) ) \right)$.
\end{lemma}

%

\medskip
In the {\it penalization approach}, we replace the functional $\phi_\eps(E)$, which is to be maximized (+) or minimized (-), by the penalized functional to be maximized, 
$$\psi_{\eps,c} (E) = \pm\phi_\eps(E) - c_A \phi_{A,\eps}(E) - c_D \phi_{D,\eps}(E)
$$ 
with  large parameters ${c_A,c_D\gg 1}$.
The differential equation \eqref{ode-E} is then replaced by 
\begin{equation}
\label{ode-Ec}
\dot{E} =  \pm G_\eps(E) - c_A G_{A,\eps}(E) - c_D G_{D,\eps}(E) - \mu E,
\end{equation}
where the Lagrange multiplier $\mu = \langle \pm G_\eps(E) - c_A G_{A,\eps}(E) - c_D G_{D,\eps}(E), E \rangle$ is chosen to ensure that 
$\| E \|_F$ is conserved along solution trajectories.

Alternatively, we can use the {\it constrained gradient system}
\begin{equation}
\label{ode-E-con}
\dot{E} = \pm G_\eps(E)  - \mu_A G_{A,\eps}(E) - \mu_D G_{D,\eps}(E) - \mu E,
\end{equation}
where the Lagrange multipliers $\mu,\mu_A,\mu_D$ are chosen such that the constraints
\begin{equation}
\langle E, \dot E \rangle =0,\quad\ 
\langle G_{A,\eps}(E), \dot E \rangle \le 0,\quad\ 
\langle G_{D,\eps}(E), \dot E \rangle \le 0
\end{equation}
are satisfied. Here we note that the latter two constraints become inactive when $\Re \lambda_{\rm max} \bigl(A + \eps L_A[E] \bigr)\le -\vartheta_A$
and $\lambda_{\rm min}\bigl({\rm sym} (D + L_D[\eps E])\bigr)\ge \vartheta_D$.
\eng

\section{The equation for the perturbation size $\eps$ for Problem~1 (passivation)}
\label{sec:eps}
We consider Problem~1, which corresponds to the $+$ sign above.

In order to present the outer iteration for optimizing the perturbation size $\eps$, we need to study the behavior of a pair of non-imaginary eigenvalues close to coalescence on the imaginary axis.

Let $E(\eps)$ of unit Frobenius norm be a local maximizer of the optimization problem \eqref{eig-opt}. We let $\lambda(\eps)$  denote the eigenvalue of smallest positive real part \bcl (among these we choose the one with largest imaginary part) \ecl  of $M(X+L[\eps E(\eps)])$ and $x(\eps)$ and $y(\eps)$ are the corresponding left and right eigenvectors, normalized according to \eqref{eq:scalyx}. 
We let $\oeps$ denote the smallest value of $\eps$ such that 
$$
f(\eps):= \phi_\eps(\eps E(\eps))= \Re \lambda(\eps)
$$
becomes zero and is nonzero to the right of $\eps$:
$$
f(\eps)=0 \quad\ \text{for } 0 < \eps \le \oeps \quad\ \text{ and } \quad f(\eps)>0 \quad\text{for $\eps>\oeps$ near $\oeps$.}
$$
For a given small threshold $\delta>0$, we denote by $\eps_\delta>\oeps$ the smallest value of $\eps$ such that $f(\eps)$ equals $\delta\,$:
$$
f(\eps_\delta)=\delta.
$$
To determine $\eps_\delta$, we are thus left with a one-dimensional root-finding problem. 
\bcl
This can be solved by a variety of methods, such as bisection. We aim for a locally quadratically convergent Newton-type method, which can be justified under additional assumptions that appear to be usually satisfied. If these assumptions are not met, we can always go back to bisection. The algorithm proposed in the next section actually uses a combined Newton-type/bisection approach.
\ecl

In the following we show that under additional assumptions, 
the function $f$ is differentiable to the right of $\oeps$ and has a square-root behavior $f(\eps)\sim \sqrt{\eps-\oeps}$ as $\eps \searrow \oeps$.

\begin{assumption}
For $\eps$ in a right neighborhood  of $\oeps$, i.e., $\eps\in(\oeps, \bar\eps)$ for some ${\bar\eps>\oeps}$, 
we assume that the eigenvalue $\lambda(\eps)$  of smallest positive real part with nonnegative imaginary part of the Hamiltonian matrix $M(\eps):=M(X+L[\eps E(\eps)])$  
is unique and is a \emph{simple} eigenvalue. Moreover,  $E(\eps)$ and $\lambda(\eps)$  are assumed to be smooth functions of~$\eps$.
\label{assumpt}
\end{assumption}

Under Assumption \ref{assumpt}, also the associated normalized eigenvectors $x(\eps), y(\eps)$ with $x(\eps)^*y(\eps)>0$ are smooth functions of $\eps$ for $\eps\in(\oeps, \bar\eps)$.

\bng
%
%
We remark that Assumption \ref{assumpt} is not essential to our algorithm; if such an assumption were not
fulfilled, the derivative formula in Newton's iteration would not hold and the algorithm would instead 
automatically turn to a bisection technique. 
In our numerical experiments we verified the validity of the derivative formula which is based on 
Assumption \ref{assumpt}, and thus obtained a faster convergence with respect to a pure bisection, due to the 
quadratic convergence of the Newton iteration. 
\eng

\bcl
We make a further assumption that is needed for the Newton iteration, as it will ensure that $f'(\eps)$ has no zeros close to $\oeps$.
\ecl

\begin{assumption}
\label{assumpt-G} 
We assume that the gradient $G(\eps):= G_\eps(E(\eps))$ is different from zero for $\eps\in(\oeps, \bar\eps)$.
\end{assumption}


The following result provides an inexpensive formula for the computation
of the derivative of $\eps \mapsto \phi_\eps(E(\eps))$, which will be used in a Newton-type outer iteration of the method. 
We denote
$$
f(\eps) =\phi_\eps(E(\eps)), \quad\
G(\eps) = G_\eps(E(\eps)), \quad\
\kappa(\eps)=\frac 1{x(\eps)^* y(\eps)}.
$$
\begin{lemma}
\label{lem:der}
Under Assumptions~{\rm \ref{assumpt}} and~{\rm \ref{assumpt-G}}, 
the function $f$ is differentiable and monotonically increasing on the interval $(\oeps, \bar\eps)$, and its derivative satisfies (with ${\phantom{a}'}= d/d\eps$)
\begin{equation}
f'(\eps)  =  \kappa(\eps)\,\| G(\eps) \|_F.
\label{eq:derFdeps}
\end{equation}
\end{lemma}
\medskip
\begin{proof} Differentiating $f(\eps)=\phi_\eps\bigl( E(\eps) \bigr)$ with respect to $\eps$ we obtain, by the same calculation that led to Lemma~\ref{lem:gradient},
\begin{eqnarray}
\hskip -9mm
f'(\eps) & = & \kappa(\eps)\,
\bigl\langle G(\eps), E(\eps)+\eps E'(\eps) \bigr\rangle.
\label{eq:dFdeps}
\end{eqnarray}
From here on the proof is completed by observing that $|\bigl\langle G(\eps), E(\eps) \bigr\rangle| =  \| G(\eps) \|_F$ and 
$\bigl\langle G(\eps),  E'(\eps) \bigr\rangle =0$, as in Lemma 3.2 in \cite{Guglielmi2017}.
This yields $|f'(\eps)|  =  \kappa(\eps)\,\| G(\eps) \|_F$. Since  $f(\oeps)=0$ and $f(\eps)>0$ for $\eps>\oeps$ near $\oeps$, the derivative $f'$ must somewhere be positive. Since $\kappa(\eps)>0$ and $G(\eps)\ne 0$ by Assumption, $f'$ cannot change sign and is therefore positive for all $\eps \in (\oeps, \bar\eps)$.
\end{proof}

\bcl
Under a further assumption, we will be able to study the asymptotic behavior of $f(\eps)$ as $\eps\searrow\oeps$ beyond the mere fact
that $f(\eps)=\Re \lambda(\eps) \searrow 0$ as $\eps\searrow\oeps$.

\begin{assumption} \label{assumpt-epsstar}
We assume that the limit $M(\oeps):=\lim_{\eps\searrow\oeps} M(\eps)$ of the Hamiltonian matrices exists and that the purely imaginary eigenvalue
$\lambda(\oeps)=\lim_{\eps\searrow\oeps} \lambda(\eps)$ of $M(\oeps)$ has algebraic multiplicity two and is defective (that is, the zero singular value of $M(\oeps)-\lambda(\oeps) I$ is simple).
\end{assumption}

By definition of $\oeps$, the eigenvalue $\lambda(\oeps)$ is on the imaginary axis and has even multiplicity
because of the symmetry of the eigenvalues with respect to the imaginary axis. Here we assume multiplicity two. 
There are important classes of systems where the defectivity of the eigenvalue on the imaginary axis is a known fact.
It follows from \cite[Theorem 3.6]{KotBBP19} that the eigenvalue $\lambda(\oeps)$ is defective if the limit system $\widehat X=X+L[\oeps E(\oeps)]$ has a positive real transfer function and a positive definite matrix $\widehat D+\widehat D^\top$, and if the eigenvalue $\lambda(\oeps)$ is the only eigenvalue of $M_p(\widehat X)$ on the imaginary axis and is controllable. We do not know of an analogous result for the bounded real case.
\ecl

\bng

Assumption \ref{assumpt-epsstar} has partial relevance to the outer iteration. It has no role in Algorithm
\ref{alg_dist} (treating the case where $\delta$ is not very small), but only in Algorithm \ref{algo},
which implements the outer iteration in the case where $\delta \ll 1$. 
However,  Assumption \ref{assumpt-epsstar} is not essential to Algorithm \ref{algo}.
If the assumption were not fulfilled, the algorithm would not make use of the 
iteration \eqref{eq:stepk}-\eqref{eq:stepk2} that converges quadratically under Assumption \ref{assumpt-epsstar}, and would instead automatically reduce to a simple 
bisection technique. 
In our numerical experiments we have always observed the validity of the assumption and obtained
a faster convergence than with a pure bisection due to the quadratically convergent behavior
of the iterative method \eqref{eq:stepk}-\eqref{eq:stepk2}. 

\eng
\bcl
Under Assumption \ref{assumpt-epsstar}, the eigenvalue $\lambda(\oeps)$ of $M(\oeps)$ is non-derogatory, that is, only a single Jordan block corresponds to this eigenvalue, and hence its left and right eigenspaces are of dimension 1. Since $\lambda(\oeps)$ is a defective 
eigenvalue, left and right eigenvectors at $\oeps$ are orthogonal to each other: $x(\oeps)^*y(\oeps) =0$.

We will need the following result.

\begin{theorem} \label{thm:yJx} 
Under Assumptions~{\rm \ref{assumpt}}--\,{\rm \ref{assumpt-epsstar}} and under a nondegeneracy condition on eigenvectors of $M(\eps)$ stated in \eqref{nondeg-condition} below, 
there exist left and right eigenvectors $x(\eps)$ and $y(\eps)$ to the eigenvalue
$\lambda(\eps)$ of the Hamiltonian matrix $M(\eps)$, normalized to norm 1 with 
 $x(\eps)^*y(\eps)>0$ for $\eps>\oeps$, that depend continuously on $\eps$ in the closed interval $[\oeps,\bar\eps]$ (in particular, they converge for $\eps\searrow \oeps$).
In the limit we have
\[
y(\oeps) = \pm Jx(\oeps),
\]
where the sign depends on $x(\eps)$ for $\eps$ near $\oeps$.
\end{theorem}
\ecl

\begin{proof}
%
%
%
\bng
By \cite[Theorem 5.1]{Paige81}, 
since $M(\eps)$ has no imaginary eigenvalues for $\eps > \oeps$,
it admits a real Schur-Hamiltonian decomposition, i.e., there exists an orthogonal symplectic real matrix 
$S(\eps)$ (that is, $S(\eps)^\top S(\eps)=I$ and $S(\eps)^\top J S(\eps)=J$) for $\eps > \oeps$ that \eng\bcl transforms $M(\eps)$ to a block triangular Hamiltonian matrix
\ecl
\begin{equation}
M_0(\eps) = S(\eps)^{-1} M(\eps) S(\eps)
= \begin{pmatrix} F(\eps) & H(\eps) \\ 0 & -F(\eps)^\top \end{pmatrix} ,
\end{equation}
\bng
where $H(\eps)$ is symmetric and $F(\eps)$ is upper quasi-triangular.  
\eng

For the eigenvalue $\lambda(\eps)$, the left and right eigenvectors of $M_0(\eps)$ are related to those of $M(\eps)$ by
\begin{equation} \label{x-x0}
x_0(\eps) = S(\eps)^\top x(\eps), \qquad y_0(\eps) = S(\eps)^{-1} y(\eps).
\end{equation}
We assume that $x(\eps)$ and $y(\eps)$ are normalized to norm 1 and such that $x(\eps)^*y(\eps)>0$ for $\eps>\oeps$, and hence we have also 
\begin{equation} \label{x0y0pos}
\text{$x_0(\eps)$ and $y_0(\eps)$ are of norm 1 and\ }\
x_0(\eps)^*y_0(\eps)>0.
\end{equation}
Noting that the lower half of the right eigenvector $y_0(\eps)$ to the block triangular matrix $M_0(\eps)$ consists only of zeros, 
we split the eigenvectors  into the upper and the lower $n/2$-dimensional subvectors as
\begin{equation}
y_0(\eps) = \left( \begin{array}{r}  -p(\eps) \\ 0\ \ \end{array} \right), \qquad 
x_0(\eps) = \left( \begin{array}{r} -s(\eps) \\ r(\eps) \end{array} \right) .
\end{equation}
By the Hamiltonian symmetry, we can represent the eigenvectors associated with the eigenvalue ${}-\conj{\lambda}(\eps)$ as
$\widetilde y_0(\eps) = J x_0(\eps)$, $\widetilde x_0(\eps) = J y_0(\eps)$.
This gives
\begin{equation}
\widetilde y_0(\eps) =  \left( \begin{array}{r} r(\eps) \\ s(\eps) \end{array} \right), \qquad 
\widetilde x_0(\eps) = \left( \begin{array}{r} 0\ \ \\ p(\eps) \end{array} \right). 
\end{equation}
\bcl
By compactness, there exists a sequence $(\eps_n)$ with $\eps_n \searrow \oeps$ as $n\to\infty$ such that $x_0(\eps_n)$, $y_0(\eps_n)$ and $S(\eps_n)$ converge to vectors $\widehat x_0$, $\widehat y_0$ of norm 1 and an  orthogonal symplectic real matrix $\widehat S$. By the continuity of $M(\cdot)$ and $\lambda(\cdot)$ at $\oeps$, the limit vectors $\widehat x_0$, $\widehat y_0$ are then left and right eigenvectors  corresponding to the purely imaginary eigenvalue $\lambda(\oeps)$ of $M(\oeps)$.

By Assumption~\ref{assumpt-epsstar}, the left and right eigenspaces to $\lambda(\oeps)$ are one-dimensional, and so we have that for some complex $\xi,\eta$ of unit modulus,
\begin{equation}\label{y0-x0-limits}
\lim\limits_{n\to\infty} \widetilde{y}_0(\eps_n) = -\eta \lim\limits_{n\to\infty} y_0(\eps_n), \qquad
\lim\limits_{n\to\infty} \widetilde{x}_0(\eps_n) = \xi \lim\limits_{n\to\infty} x_0(\eps_n) .
\end{equation}   
We thus obtain 
\begin{equation}\label{s-to-zero}
\lim_{n\to\infty}  s(\eps_n) = 0
\end{equation} 
and
\begin{equation} \label{eq:limits}
\lim\limits_{n\to\infty}  r(\eps_n) = \eta \lim\limits_{n\to\infty} p(\eps_n), \qquad
\lim\limits_{n\to\infty}  p(\eps_n) = \xi \lim\limits_{n\to\infty} r(\eps_n) ,
\end{equation} 
\ecl
so that
\begin{equation}
\xi = \bar\eta.
\end{equation}
By \eqref{x0y0pos},
\begin{equation} \label{sp-pos}
s(\eps)^*p(\eps) \ \text{ is real and positive for } \eps > \oeps,
\end{equation}
\bcl
and in particular, $s(\eps)\ne 0$ for $\eps > \oeps$ (but recall \eqref{s-to-zero}).
\ecl
As is noted in \cite[(1.4)]{Paige81},
\begin{equation} \label{sr-real}
s(\eps)^* r(\eps) \ \text{ is real for } \eps > \oeps.
\end{equation}
\bcl
Under the nondegeneracy condition
\begin{equation}\label{nondeg-condition}
\liminf_{\eps\searrow\oeps} \,\left| \left(\frac{ s(\eps) }{ \|s(\eps)\| }\right)^* \frac{ r(\eps) }{ \|r(\eps)\| } \right| > 0,
\end{equation}
which states that the normalizations of the vectors $s$ and $r$ are not asymptotically orthogonal,
we conclude that there is a subsequence $(\eps_n')$ of $(\eps_n)$ such that the normalized sequence $\bigl(s(\eps_n')/\|s(\eps_n')\|\bigr)$ is convergent and (on noting $\|r(\eps_n)\|\to 1$)
\begin{equation}\label{sp-conv}
\lim_{n\to\infty} \frac{ s(\eps_n')^*r(\eps_n') }{ \|s(\eps_n')\| } \ne 0.
\end{equation}
As \eqref{eq:limits} implies that this nonzero limit equals
$$
\lim_{n\to\infty} \frac{ s(\eps_n')^*r(\eps_n') }{ \|s(\eps_n')\| } =
\eta \lim_{n\to\infty} \frac{ s(\eps_n')^*p(\eps_n') }{ \|s(\eps_n')\| }  
$$
and the two limits in this formula are real by \eqref{sp-pos} and \eqref{sr-real}, 
it follows that $\eta$ is real and hence $\eta$ equals $1$ or $-1$. In view of $\eqref{sp-pos}$ and \eqref{nondeg-condition}, we actually have
\begin{equation}\label{eta-sr}
\eta=\lim_{\eps\searrow\oeps}\ \mathrm{sign}( s(\eps)^* r(\eps)) = \pm 1,
\end{equation}
which depends only on the left eigenvector $x_0(\eps)$.
As a consequence, we obtain from \eqref{y0-x0-limits} that
\begin{equation} \label{eq:limy0}
\widehat y_0 = -\eta J \widehat x_0= \mp J \widehat x_0.
\end{equation}
By \eqref{x-x0} we have
\begin{equation}
x(\eps) =  \left( S(\eps)^\top \right)^{-1} x_0(\eps) =  - J S(\eps) J x_0(\eps), \qquad y(\eps) = S(\eps) y_0(\eps) = \pm S(\eps) J x_0(\eps)
\end{equation}
and therefore the limits $\widehat x = \lim_{n\to\infty} x(\eps_n)$ and $\widehat y = \lim_{n\to\infty} y(\eps_n)$ exist and satisfy
\begin{equation} \label{eq:limy}
\widehat y = \pm J \widehat x.
\end{equation}
We now use once again that by Assumption~\ref{assumpt-epsstar}, the left and right eigenspaces to $\lambda(\oeps)$ are one-dimensional. 
Hence $\widehat x$ is a complex multiple of the unique left eigenvector $x(\oeps)$ of norm 1 for which the first nonzero entry is positive.
If we choose the eigenvectors $x(\eps)$ such that their corresponding entry is also nonnegative, then we find that every convergent subsequence
$(x(\eps_n))$ converges to the same limit $x(\oeps)$ as $n\to\infty$, and hence $x(\eps)$ converges to $x(\oeps)$ as $\eps\searrow\oeps$. To the left eigenvector $x(\eps)$, there corresponds a unique right eigenvector $y(\eps)$ of norm 1 that satisfies
$x(\eps)^*y(\eps)>0$ for $\eps>\oeps$. By \eqref{eq:limy}, the limit of every convergent subsequence $(y(\eps_n))$ converges to
$\pm J \lim_{n\to\infty} x(\eps_n)=\pm Jx(\oeps)$, and hence the limit $y(\oeps):= \lim_{\eps\searrow\oeps} y(\eps)$ exists, is a right eigenvector of $M(\oeps)$ to the eigenvalue $\lambda(\oeps)$, and it equals
$$
y(\oeps) = \lim_{\eps\searrow\oeps} y(\eps) = \pm J \lim_{\eps\searrow\oeps} x(\eps) = \pm J x(\oeps),
$$
which completes the proof.
\ecl
\end{proof}

\bcl
We are now in a position to characterize the asymptotic behavior of the function $f(\eps)=\Re\lambda(\eps)$ as $ \eps \searrow \oeps$.
\ecl

\begin{theorem}\label{thm:sqrt}
Under the assumptions of Theorem~{\rm \ref{thm:yJx}} and if $G(\oeps)=\lim_{\eps\searrow\oeps} G(\eps)$ exists and is different from $0$,
we have
$$
\Re \lambda(\eps)= \gamma \,\sqrt{\eps-\oeps} \;(1+o(1)) \quad\text{ as } \eps \searrow \oeps
$$
for some positive constant $\gamma$.
\end{theorem}

\begin{proof} 
The proof adapts the proof of Theorem 5.2 in \cite{Guglielmi2015} to the current situation.
We consider the nonnegative function 
$$
\vartheta(\eps) := \frac1{\kappa(\eps)}= x(\eps)^* y(\eps) > 0 \ \text{ for }\ \eps \in (\oeps, \bar\eps), \qquad
\vartheta(\oeps) = 0.
$$ 
Based on results in \cite{MS88,GO11}, the derivative $\vartheta'(\eps)$ is obtained in part (a) of the proof of Theorem 5.2 in \cite{Guglielmi2015} as
\begin{eqnarray}
\vartheta'(\eps) & = & x(\eps)^* M'(\eps) Z(\eps) x(\eps) \vartheta(\eps) +
y(\eps)^* Z(\eps) M'(\eps) y(\eps) \vartheta(\eps) ,
\label{eq:derdelta}
\end{eqnarray} 
where the group inverse $Z(\eps)$ of $N(\eps):=M(\eps)-\lambda(\eps)I$ is related to
the pseudoinverse $N(\eps)^\dag$ by the formulas 
\begin{eqnarray}
Z(\eps) & = & \frac{1}{\vartheta(\eps)^2} \widehat{Z}(\eps)
\nonumber
\\[1mm]
\widehat{Z}(\eps) & = & \bigl( \vartheta(\eps) I - y(\eps) x(\eps)^* \bigr) 
N(\eps)^\dag
\bigl( \vartheta(\eps) I - y(\eps) x(\eps)^* \bigr).
\label{eq:Ghat}
\end{eqnarray} 
\bng
By Assumption \ref{assumpt-epsstar} the second smallest
singular value $\sigma_2(\eps)$ of $M(\eps)-\lambda(\eps)I$
does not converge to zero; as a consequence $
N(\eps)^\dag$ remains bounded as $\eps \searrow \oeps$.
\eng
Thus
\begin{eqnarray}
\widehat{Z}(\eps) & = & 
y(\eps) x(\eps)^* N(\eps)^\dag y(\eps) x(\eps)^* + O(\vartheta(\eps)) 
\\
&=& \nu(\eps) y(\eps)x(\eps)^* + O(\vartheta(\eps)) 
\nonumber
\end{eqnarray}
with the factor
\begin{eqnarray*}
 \nu(\eps) & := & x(\eps)^* N(\eps)^\dag y(\eps),
\label{eq:kappa}
\end{eqnarray*} 
which has $ \nu(\oeps)\ne 0$ by part (b) of the proof of Theorem 5.2 in \cite{Guglielmi2015}. 
Furthermore, we set
$$
\mu(\eps) := x(\eps)^* M'(\eps) y(\eps),
$$
for which we note that $\mu(\eps)=  \lambda'(\eps)\vartheta(\eps)$ and hence, by Lemma~\ref{lem:der}, $\Re\mu(\eps)=\|G(\eps)\|_F$,
which has a well-defined nonzero limit as $\eps \searrow \oeps$.

We insert the expression for the group inverse $Z(\eps)$ into \eqref{eq:derdelta} and
note that $x(\eps)^*M'(\eps)N^\dagger(\eps)x(\eps)=0$ and $y(\eps)^*N^\dagger(\eps)M'(\eps)y(\eps)=0$ because of
$N^\dagger(\eps)x(\eps)=0$ and $y(\eps)^*N^\dagger(\eps)=0$, which follows from $x(\eps)^*N(\eps)=0$ and $N(\eps)y(\eps)=0$, respectively.
We then obtain
\begin{align}\nonumber
\vartheta'(\eps)\vartheta(\eps) &= x(\eps)^* M'(\eps) \widehat Z(\eps) x(\eps)  +
y(\eps)^* \widehat Z(\eps) M'(\eps) y(\eps)  
\\
&= 2 \, \nu(\eps) \mu(\eps) + O(\vartheta(\eps)\mu(\eps)).
\label{vartheta-prime}
\end{align}

In the limit $\eps\searrow\oeps$ we have,
recalling that $y(\oeps)= \pm Jx(\oeps)$  by Theorem~\ref{thm:yJx}  and noting that $JN(\oeps)$ is a hermitian matrix,
$$
 \nu(\oeps) = x(\oeps)^* N(\oeps)^\dag y(\oeps) = x(\oeps)^* (JN(\oeps))^\dag  J y(\oeps) 
= \mp \, x(\oeps)^* (JN(\oeps))^\dag x(\oeps) \in \R,
$$
which implies that
$$
\mu(\eps) (1+O(\vartheta(\eps)) =  \frac{\vartheta'(\eps)\vartheta(\eps)}{2  \nu(\eps)} \in \R.
$$
Using the fact that $\lim_{\eps \searrow \oeps} \Re\mu(\eps)=\|G(\oeps)\|_F$ exists and is nonzero by assumption, this implies that
$$
\mu(\oeps) := \lim_{\eps \searrow \oeps} \mu(\eps) \in \R \quad\text{exists and is nonzero.}
$$
Hence, the right-hand side of \eqref{vartheta-prime} has a nonzero finite real limit as $\eps \searrow \oeps$. 
\bcl
So we have
$$
\frac d{d\eps} \vartheta(\eps)^2 = 2 \vartheta'(\eps)\vartheta(\eps) = 4\mu(\oeps)\nu(\oeps) (1+o(1)) \quad\text{ as }\ \eps\searrow\oeps.
$$
Integrating this relation, recalling $\vartheta(\oeps)=0$ and taking the square root (recall that $\vartheta(\eps)>0$ for $\eps>\oeps$) yields 
\begin{equation}\label{vartheta-asymptote}
\vartheta(\eps) = \sqrt{\eps-\oeps} \ 2\sqrt{ \mu(\oeps)\nu(\oeps)} \,(1+o(1)).
\end{equation}
This further allows us to conclude that $\nu(\oeps)$ is not only nonzero and real but actually positive.

On the other hand, since $\mu(\eps)=  \lambda'(\eps)\vartheta(\eps)$, we find
$$
\Re \lambda'(\eps) = \frac{\Re \mu(\eps)}{\vartheta(\eps)} = \frac{\mu(\oeps)}{\vartheta(\eps)} (1+o(1)) = 
\frac{\vartheta'(\eps)}{\nu(\oeps)} (1+o(1)),
$$
where we used once again  \eqref{vartheta-prime} in the last equality. By integration this implies 
$$
\Re \lambda(\eps) = \frac{\vartheta(\eps)}{\nu(\oeps)} (1+o(1)),
$$
which together with \eqref{vartheta-asymptote} yields the stated result.
\ecl
\end{proof}

\bigskip
\def\oeps{\widehat \eps} 
\def\aleps{\phi_{\eps}}

\section{An algorithm to solve Problem 1 (passivity enforcement)} 
\label{sec:alg1}
To recap, Problem 1 is as follows: given a matrix $X$ such that $M(X)$ has some eigenvalues on the imaginary axis and given a threshold $\delta > 0$, find a nearest 
(in the Frobenius-norm) structured matrix $X+L[\eps E]$ (with $\|E\|_F=1$ and perturbation size $\eps>0$) such that the eigenvalues of $M(X+L[\eps E])$ closest to the imaginary axis are $\delta$-close to it.

We start from an initial perturbation $\eps_0 E_0$ such that 
\[
M(X + L[\eps_0 E_0]) \qquad \mbox{has no eigenvalues on the imaginary axis.} 
\]
This can be obtained by applying methods available in the literature (see e.g. \cite{Grivet2015}). 
We thus have
\[
\phi_{\eps_0}(E_0) > 0, 
\]
where
\[
\phi_\eps(E)= \Re \lambda(M(X+ L[\eps E])) 
\]
with $\lambda(M)$ denoting the leftmost eigenvalue of $M$ in the right complex half-plane
$\C^+ = \{ z \in \C \,|\, \Re z \ge 0\}$, chosen with largest imaginary part.

\medskip
Our aim is to optimize the perturbation size $\eps$.
In particular the algorithm has the goal to compute
\begin{equation} \label{eq:defepsstar}
\widehat \eps_\delta = \inf \{\eps\ge 0: \phi_\eps(E(\eps)) = \delta  \},
\end{equation}
where
\begin{eqnarray*}
E(\eps) & = & \arg\max\limits_{ \| E \|_F = 1} \phi_\eps(E).
\end{eqnarray*}

Repeated use of Algorithm~\ref{alg_step}, which we are going to present in Section \ref{sec:fullrank}, is used to determine the extremizer $E$ such that the leftmost  
eigenvalue of $M(X+L[\widehat \eps_\delta E] )$ in the right half-plane has real part $\delta$.

Since Algorithm~\ref{alg_step} finds {\it locally} leftmost eigenvalues, the algorithms we present here 
are guaranteed only to find upper bounds on $\widehat \eps_{\delta}$; nevertheless they typically 
find good approximations for the cases we have tested. 

\subsection{Integration of the constrained gradient system} \label{sec:fullrank}
We assume that $M(X + L[\eps_0 E_0])$ has no imaginary eigenvalues.
For problems of moderate size we integrate the differential equation   
\begin{equation}
\label{ode-EP1}
\dot{E} = +  G_\eps(E)  - \mu E, \qquad \mu = \langle G_\eps(E), E \rangle.
\end{equation}
which - as we have seen - is the constrained gradient systems for the functional $\phi_\eps(E)$.


In the sequel we use the following notation: all quantities written as $g(\eps)$, like $\lambda(\eps),
E(\eps)$ and so on, are intended to be associated with stationary points (i.e., local extremizers)  of (\ref{ode-EP1}).

We discretize the differential equation \eqref{ode-EP1}  by a projected explicit Euler method with an adaptively chosen stepsize.
At each discretization step, we require -- according to the monotonicity property  of Theorem~\ref{thm:stat} --
that the real part of the leftmost  eigenvalue (of positive real part) $\lambda$ of $X+L[\eps E]$ be decreased, for a given $\eps>0$.

In this way the method determines a sequence $\left(\lambda_n,E_n\right)$
such that \bcl $\Re\,\lambda_n > \Re\,\lambda_{n-1}$, \ecl until $E_n$ approaches a stationary point.
\vspace{-1mm}
\subsubsection*{Euler step of integration} 
Given $E_n$ of unit Frobenius norm,
and given  left and right eigenvectors $x_n$ and $y_n$ of $M(X+L[\eps E_n])$
associated with its leftmost eigenvalue $\lambda_n$ of positive real part, with $x_n^*y_n>0$, we determine $E_{n+1}$ of unit Frobenius norm at time $t_{n+1} = t_n + h_{n}$ by applying a projected Euler step to \eqref{ode-EP1}; see Algorithm~\ref{alg:1}. 

\IncMargin{1em} 
\begin{algorithm}\label{alg:1}
\DontPrintSemicolon
\KwData{$E_n, x_n, y_n, \lambda_n$, $\gamma$ and $\rho_{n}$ (step size predicted by the previous step)}
\KwResult{$E_{n+1}, x_{n+1}, y_{n+1}, \lambda_{n+1}$ and $\rho_{n+1}$}
\Begin{
\nl Set $h=\rho_{n}$, \  $G_n = G(\eps E_n) = L^* \Bigl(M'(X+L[\eps E_n])^*[\Re(x_n y_n^*)]\Bigr)$, \ $\mu_n = \langle G_n, E_n \rangle$\;
\nl Compute
\vspace{-3mm}
\begin{eqnarray*}
\widetilde{E}_{n+1} &=& E_{n} + h  \bigl( G_n  - \mu E_n \bigr) ,
\quad E_{n+1} = {\widetilde{E}_{n+1}}/{\| \widetilde{E}_{n+1} \|_F}
\nonumber
\end{eqnarray*}
\nl Compute the eigenvalue $\widetilde\lambda$ of smallest positive real part (with largest imaginary part) and the corresponding left and right eigenvectors $\widetilde{x}$ and $\widetilde{y}$ of
      $M(X+L[\eps E_{n+1}])$\;
\nl \eIf{$\Re  \widetilde\lambda  \le \Re  \lambda_n $} {
    reject the step, reduce the step size as $h:=h/\gamma$ and
    repeat from 3\;
    }{
    accept the step: set $h_{n+1}= h$, $\lambda_{n+1} = \widetilde{\lambda}$,
    $x_{n+1} = \widetilde{x}$ and $y_{n+1} = \widetilde{y}$\;
    }
\nl \eIf{$h_{n+1} = \rho_{n}$} {increase the stepsize as $\rho_{n+1} := \gamma \rho_{n}$\;
    }{
    set $\rho_{n+1} = \rho_{n}$\;
    }
\nl Proceed to next step\;
}
\caption{Projected Euler step, full rank integration}
\label{alg_step} 
\end{algorithm}

The stepsize control only aims at reducing the real part of the eigenvalue $\lambda$.
As we are only interested in reaching a stationary point of the differential equation, 
an accurate approximation of the exact trajectory $E(t)$ is not of interest here.

The cost of one step of the algorithm is dominated by the computation of leftmost eigenvalues with positive real part
and associated left and right eigenvectors of a Hamiltonian matrix, or equivalently, of the corresponding matrix pencil; see Section~\ref{subsec:intro-ham}.
\subsubsection*{Initial values} \
In the beginning we assume to be provided, by some existing code, a  matrix $\eps_0 E_0$ such that 
$M(X + L[\eps_0 E_0])$ has no imaginary eigenvalues. 
\bcl
It is not necessary that $\eps_0 E_0$ be close to the optimal perturbation. It is, however, advantageous for the following reasons:\\
(i) If $E_0$ is far from $E(\eps_0)$, then reaching a stationary point of (\ref{ode-EP1}) to a given tolerance is computationally more costly.
Moreover, a trajectory starting far from the global extremizer $E(\eps_0)$ is more likely to get stuck in a local optimum.
\\
(ii) If $\eps_0$ is far from $\widehat \eps_\delta$, then the outer iteration takes longer to converge.

In subsequent steps of the outer iteration, we start the differential equation (\ref{ode-EP1}) for $\eps=\eps_{\ell+1}$ with the
approximation to $E(\eps_\ell)$ computed
by Algorithm~\ref{alg_step} for $\eps = \eps_\ell$.
\ecl



\subsubsection*{Terminating at a global extremizer} \ 
In order to find a global (and not just local) extremizer, we run the algorithm with a set of different 
initial values to reduce the possibility of getting trapped in a point 
that is only locally leftmost, which cannot be excluded from following a single trajectory and which we have indeed 
occasionally observed in our numerical experiments. 
In view of the monotonicity of the real part of the eigenvalue along a trajectory (Theorem~\ref{thm:stat}), 
it appears that only exceptional trajectories would run into a stationary point that does not correspond to a locally 
leftmost point of the pseudospectrum, and in fact we never observed such a situation in our extensive numerical experiments.



\subsection{The outer iteration to compute the perturbation size} 
\label{sec:outer}

In order to compute the value of $\widehat \eps_\delta$ defined in~\eqref{eq:defepsstar}, 
starting from $\eps > 0$ such that $\phi_\eps(E(\eps)) > 0$, we want to compute a root $\widehat \eps_\delta$ of the equation
\begin{eqnarray}
f(\eps) := \phi_\eps (E(\eps)) & = & \delta.
\label{eq:adelta}
\end{eqnarray}
Since the function $\eps \mapsto \aleps(E(\eps))$ is monotonically increasing in a right neighborhood of
the minimal solution $\widehat \eps_\delta$ of \eqref{eq:adelta}, and since we have a formula for the derivative of
$\aleps(E(\eps)$ with respect to $\eps$,
we can apply a Newton/bisection method to compute a zero of Equation \eqref{eq:adelta}.

%

\subsubsection*{The case when $\delta$ is not too small: a Newton/bisection method}
For conciseness we omit the dependence on $E$ in the coded algorithms 
and denote it simply by $\aleps$.
We consider a given $\delta > 0$.

\bcl
In view of Lemma \ref{lem:der}, which gives the derivative of $f$, we make use of a Newton iteration,
\[
\eps_{n+1} = \eps_{n} - 
\frac{1}{\kappa(\eps_n) \| G(\eps_n) \|_F}\, \left(f(\eps_n)-\delta \right),
\qquad n=0,1,\ldots,
\]
which is coupled with bisection in Algorithm \ref{alg_dist}.
The method yields quadratic convergence to $\widehat \eps_\delta$ under Assumptions~\ref{assumpt} and~\ref{assumpt-G}.
\ecl

\IncMargin{1em} 
\begin{algorithm}
\DontPrintSemicolon
\KwData{Matrix $X$, $\eps_0$, $E_0$, $k_{\max}$ (max number of iterations), tol (tolerance)\; 
$\eps_{\rm lb}$ and $\eps_{\rm ub}$ (starting values for the lower and upper 
bounds for $\widehat \eps_{\delta}$)} 
\KwResult{$\widehat \eps_\delta$ (upper bound for the measure)}
\Begin{
\nl Solve the ODE (\ref{ode-EP1}) with $\eps=\eps_0$ and $E(0)=E_0$\;
\nl Set $\lambda(\eps_0)$ leftmost eigenvalue of $M(X + L[\eps_0 E(\eps_0)] )$ in the right complex half-plane, 
$x(\eps_0)$ and $y(\eps_0)$ 
left and right eigenvectors\; 
\nl Compute $f(\eps_{0})$, $f'(\eps_0)$\;
\nl \bcl Set $\eps_1 = - (f(\eps_{0})-\delta)/f'(\eps_{0})$\;  \ecl
\nl Set $k=0$\;
\nl Initialize lower and upper bounds: $\eps_{\rm lb}=0$, $\eps_{\rm ub}=+\infty$\; 
\While{$|f(\eps_k)-f(\eps_{k-1})| < {\rm tol}$}{
\nl Solve the ODE (\ref{ode-EP1}) with $\eps=\eps_k$ and  $E(0)=E(\eps_{k-1})$\;
\nl Set $\lambda(\eps_k)$ leftmost eigenvalue of $M(X + L[\eps_k E(\eps_k)] )$ in the right complex half-plane, 
$x(\eps_k)$ and $y(\eps_k)$ 
left and right eigenvectors\; 
\nl Compute $f(\eps_k)$, $f'(\eps_k)$\;
\nl Update upper and lower bounds $\eps_{\rm lb}$, $\eps_{\rm ub}$\; 
\nl \eIf{$f(\eps_k) = 0$} 
{
Set $ \eps_{\rm ub} = \min(\eps_{\rm ub},\eps_k)$\;
Compute $\eps_{k+1} = (\eps_{\rm lb} + \eps_{\rm ub})/2$ \ (bisection step)
} 
{
Set $ \eps_{\rm lb} = \max(\eps_{\rm lb},\eps_k)$\;
\nl Compute $f'(\eps_k)$\;
\nl \bcl Compute $\eps_{k+1} = \eps_{k} - \displaystyle{(f(\eps_k)-\delta)/{f'(\eps_k})}$ \ecl \ (Newton step)
}
\eIf{$k=k_{\max}$}
{Halt}
{Set $k=k+1$}
}
\nl Return $\widehat \eps_\delta = \eps_k$\;
}
\caption{Newton/bisection method for Problem 1}
\label{alg_dist} 
\end{algorithm}



\bcl
As the Newton method relies on these assumptions, we 
combine it with a  bisection method to guarantee convergence also in nonsmooth cases.
On the other hand, in all numerical experiments we have performed, we always 
observed --- at least when $\delta$ is not very small --- a fast convergence in agreement with Assumptions~\ref{assumpt} and~\ref{assumpt-G}. 
\ecl

\subsubsection*{The case of very small $\delta$}
\bcl
For $\eps\searrow\widehat \eps$,  we have by Theorem~\ref{thm:sqrt} for $f(\eps)=\Re \lambda(\eps)$ (and by its proof for $f'(\eps)$) the square-root behavior 
\begin{eqnarray}
\begin{array}{rcl}
f(\eps) & = & \gamma \sqrt{\eps - \widehat \eps} \ (1+o(1)) 
\\[2mm]
f'(\eps) & = & \displaystyle{\frac{\gamma}{2 \sqrt{\eps - \widehat \eps}}} \ (1+o(1)) ,
\end{array}
\label{eq:eps}
\end{eqnarray}
in the expected case of a defective coalescence of two eigenvalues on the imaginary axis.
\ecl
\begin{algorithm}
\DontPrintSemicolon
\KwData{$\delta$, ${\rm tol}$, $\theta$ (default $0.8$), 
and $\eps_0$ (such that $f(\eps_0) > {\rm tol}$)} 
\KwResult{$\widehat \eps_{\delta}$}
\Begin{
\nl Set {\rm Reject} = {\rm False} and $k=0$\;
\nl \While{$|f(\eps_k) - \delta| \ge {\rm tol}$}{
\nl \eIf{${\rm Reject} = {\rm False}$} {
\nl	Set $\widetilde\eps = \eps_{k}$, \ $\widetilde{\theta}=\theta$\; 
    Compute $\gamma_k$ and $\widehat \eps_k$ by (\ref{eq:stepk})\;
\nl Set $\eps_{k+1} = \widehat \eps_k - \displaystyle{\frac{\delta^2}{\gamma_k^2}}$\;
    }{
    Set $\eps_{k+1} = \widetilde\theta\,\eps_{k}+ (1-\widetilde\theta)\,\widetilde\eps$\;
		Set $\widetilde\theta = \theta\widetilde\theta$\;
    }
\nl Set $k=k+1$\;
\nl Compute $f(\eps_k)$ by integrating (\ref{ode-EP1}) with initial datum $E(\eps_{k-1})$\;
\nl Compute $f'(\eps_k)$ by (\ref{eq:derFdeps})\;		
\nl \eIf{$f(\eps_k) < {\rm tol}$}{ 
    Set ${\rm Reject} = {\rm True}$} {
    Set ${\rm Reject} = {\rm False}$ }
}
\nl Print $\widehat \eps_\delta \approx \eps_k$\;
\nl Halt
}
\caption{Basic algorithm for computing the optimal perturbation for small $\delta$ \label{algo}}
\end{algorithm}
For an iterative process, given $\eps_k$, we use Lemma \ref{lem:der} to compute $f'(\eps_k)$ and solve (\ref{eq:eps}) for $\gamma$ and 
$\widehat \eps$, ignoring the $o(1)$ terms.
We denote the solution as $\gamma_k$ and $\widehat \eps_k$, i.e.,
\begin{eqnarray}
\gamma_k & = & \sqrt{2 f(\eps_k) f'(\eps_k)}, \qquad
\widehat \eps_k = \eps_k - \frac{f(\eps_k)}{2 f'(\eps_k)} 
\label{eq:stepk}
\end{eqnarray}
and then compute
\begin{eqnarray}
\eps_{k+1} & = & \widehat \eps_k + {\delta^2}/{\gamma_k^2}.
\label{eq:stepk2}
\end{eqnarray} 
\bcl
Algorithm \ref{algo} is based on these formulas.
%
The test at line {\bf 9} - when positive - means that for $\eps=\eps_k$ there are
coalescing eigenvalues (up to a tolerance {\rm tol}). 
The algorithm shows quadratic convergence as we will illustrate in Section~\ref{sec:num}.
\ecl

\bigskip

\section{Problem 2: distance to the nearest non-passive system}
\label{sec:alg2}
The algorithm for computing the real passivity radius requires only few modifications of the algorithm for passivity enforcement described in the previous section.

\subsection{Integration of the constrained gradient system}
For Problem 2, the differential equation to solve in the inner iteration is given by \eqref{ode-E} with the minus sign, i.e., 
\[
\dot{E} = {}-\bigl( G_\eps(E)  - \mu E \bigr)\qquad \text{with }\ \mu = \langle G_\eps(E), E \rangle.
\]

\subsection{The outer iteration to compute the perturbation size} 
The outer iteration still pivots on the behavior of a pair of non-imaginary eigenvalues close to coalescence on the imaginary axis.
Similar to Theorem \ref{thm:sqrt}, we now have the following result.

\begin{theorem}\label{thm:sqrt2}
Under Assumptions~{\rm \ref{assumpt}--\ref{assumpt-epsstar}} (now for $0<\eps<\oeps$) 
\bcl and the non\-degeneracy condition \eqref{nondeg-condition} \ecl
we have
$$
\Re \lambda(\eps)= \gamma \,\sqrt{\oeps-\eps} \,(1+o(1)) \quad\text{ as } \eps \nearrow \oeps
$$
for some positive constant $\gamma$.
\end{theorem}

Algorithms \ref{alg_dist} and~\ref{algo} are adapted accordingly.

%
%


\section{Low-rank dynamics for large systems}
\label{sec:low-rank}

We recall that the rescaled gradient is
$G_\eps(E)= L^\ast M'(X+L[\eps E])^\ast[\Re(x y^*)] \in \R^{k\times l}$, where we note that $\Re(x y^*)$ has rank at most 2. \bcl Lemmas~\ref{lemma:formulaoperator-p} and~\ref{lemma:formulaoperator-b} then show that for $M=M_p$ and $M=M_b$, respectively, \ecl the matrix $M'(X+L[\eps E])^\ast[\Re(x y^*)]\in \R^{2n\times 2n}$ is of moderate rank independent of $n$. If $L^\ast$ maps matrices of this rank to matrices of moderate rank, say at most $r$, then also $G_\eps(E)$ is of  rank at most $r$ for any $E$. In particular, for $L$ chosen as in \eqref{L-example}, Lemmas~\ref{lemma:formulaoperator-p} and~\ref{lemma:formulaoperator-b} and \eqref{L-star} show that (with $k=p$ and $l=n$)
$$
\text{rank}\,(G_\eps(E)) \le r=8 \quad\text{ for all }E\in \R^{k\times l},
$$
independently of the dimension $n$.
In this section we describe an algorithm that makes use of this low-rank structure. This low-rank algorithm appears particularly suited for large-scale passivation problems.
Since Theorem~\ref{thm:stat} shows that in a stationary point, $E$ is proportional to $G_\eps(E)$ and hence of rank at most $r$, we {\it restrict
the dynamics for $E(t)$ to matrices of rank $r$}, which turns out to have exactly the same stationary points as the original gradient system. This  leads to an algorithm that works with time-dependent factor matrices of dimensions $k\times r$, $l\times r$, and $r\times r$ instead of solving matrix differential equations of dimension $k\times l$ as in the previous section, and moreover, the low-rank structure can be beneficial in the computations of eigenvalues and left and right eigenvectors of  the Hamiltonian matrices $M(X+L[\eps E(t)])$. This makes the low-rank approach favorable for the passivation of high-dimensional systems. The proposed algorithm and its properties are similar to those of the low-rank algorithm in \cite[Section~4]{Guglielmi2017}. 

\subsection{Rank-$r$ matrices and their tangent matrices} 
Following  \cite{Koch2007}, we first collect some basic properties. Matrices of rank $r$ form a manifold, here denoted
$$
\M_r = \{ E\in \R^{k\times l}: {\rm rank}(E)=r \}.
$$
For the computation with rank-$r$ matrices, we represent $E\in\M_r$ in a non-unique way as 
\begin{equation}\label{factorize-m}
E=USV^\top,
\end{equation}
where $U,V\in\R^{k\times l}$  have orthonormal columns and $S\in\R^{r\times r}$ is invertible.  Unlike the singular value decomposition, we do not require $S$ to be diagonal. 

The tangent space
$T_E \M_r$ then consists of all matrices
\begin{equation}\label{tangent-m}
\delta E
=  \delta U SV^\top + U \delta S V^\top + US\delta V^\top,
\end{equation}
where $U^\top\delta U \hbox{ and } V^\top\delta V$ are skew-hermitian $r\times r$ matrices, and  $\delta S$ is an arbitrary $r\times r$ matrix. Moreover, $(\delta U,\delta S,\delta V)$ are determined uniquely by $\delta E$ if one imposes the gauge constraints $U^\top\delta U=0$ and $V^\top\delta V=0$.
The orthogonal projection onto the tangent space $T_E \M_r$ is given by \cite[Lemma 4.1]{Koch2007} as
\begin{equation}\label{project-m}
P_E[G] = GVV^\top- UU^\top G VV^\top +UU^\top G.
\end{equation}

\subsection{The gradient system restricted to rank-$r$ matrices}
 We replace the matrix differential equation \eqref{ode-E} on $\R^{k\times l}$ with the projected differential equation on~$\M_r$: 
\begin{equation}
\dot{E} = \pm P_E \bigl[G_\eps(E) - \mu E\bigr], 
\qquad\hbox{with }\ \mu= \langle G_\eps(E), E \rangle.
\label{ode-E-m}
\end{equation}
The following result combines Theorems 5.1 and 5.2 of \cite{Guglielmi2017}, which apply also to the gradient flow considered here with the same proofs. The result shows, in particular, that the functional $\phi_\eps$ ascends / descends along solutions of \eqref{ode-E-m} and that the differential equations (\ref{ode-E}) and (\ref{ode-E-m}) have the same stationary points.

\begin{theorem}\label{thm:decay-m}  Along solutions $E(t)$ of the matrix differential equation \eqref{ode-E-m} on the rank-$r$ manifold $\M_r$ with an initial value $E(0)\in\M_r$ of unit Frobenius norm we have $\| E(t) \|_F =1$ for all $t$, and
$$
\pm \frac{d}{dt} \phi_\eps\bigl(E(t)\bigr) \ge 0.
$$
If $\mathrm{rank}\,(G_\eps(E(t)))= r$, then the following statements are equivalent:\\[-2mm]
\begin{enumerate}
\item $\frac{d}{dt} \phi_\eps\bigl(E(t)\bigr) = 0$.
\item $\dot E = 0$.
\item $E$ is a real multiple of $G_\eps(E)$.
\end{enumerate}
\end{theorem}

\subsection{A numerical integrator for the rank-$r$ differential equation~\eqref{ode-E-m}}
The $k\times l$ matrix differential equation \eqref{ode-E-m} can be  written equivalently as a system of  differential equations for the factors $U,S,V$; see \cite{Koch2007}. The right-hand sides of the differential equations for $U$ and $V$ contain, however, the inverse of $S$, which leads to difficulties with standard numerical integrators when $S$ is nearly singular, that is, when $E$ is close to a matrix of rank smaller than $r$. We therefore follow the alternative approach of \cite{LO14}. This uses an integration method that is based on splitting the tangent space projection $P_E$, which by \eqref{project-m} is an alternating sum of three subprojections. A time step of the numerical integrator based on the Lie--Trotter splitting corresponding to these three terms can then be implemented in the following way. Like in \cite{Guglielmi2017}, we here consider a variant of the projector-splitting integrator of \cite{LO14} such that the unit Frobenius norm is preserved.

The algorithm starts from the factorized rank-$r$ matrix $E_0=U_0S_0V_0^\top$ of unit norm at time $t_0$ and computes the factors of the approximation $E_1=U_1S_1V_1^\top$, again of unit Frobenius norm, at the next time $t_1=t_0+h$:
\begin{enumerate}
\item With $G_0=G_\eps(E_0)$, set 
\[
   K_1 = U_0S_0 - hG_0 \, V_0
\] 
and, via a QR decomposition, compute the factorization 
\[
     U_1\widehat S_1\widehat \sigma_1 =K_1
\]
with  $U_1\in \R^{k\times r}$ having orthonormal columns and $\widehat S_1\in\R^{r\times r}$ of unit Frobenius norm, and a positive scalar $\widehat\sigma_1$.
\item Set
\[
     \widetilde\sigma_0 \widetilde S_0 = \widehat S_1 +U^\top_1 \,hG_0\, V_0,
\]
where $\widetilde S_0\in\R^{r\times r}$ is of unit Frobenius norm and $\widetilde\sigma_0>0$.
\item Set
\[
    L_1=  V_0 \widetilde S_0^\top - hG_0^\top U_1
\]
and, via a QR decomposition, compute the factorization 
\[
     V_1 S_1^\top\sigma_1 = L_1, 
\]
with  $V_1\in\R^{l\times r}$ having orthonormal columns, with  $S_1\in\R^{r\times r}$ of unit Frobenius norm, and a positive scalar $\sigma_1$.
\end{enumerate}
The algorithm computes a factorization of the rank-$r$ matrix of unit Frobenius norm
\[
     E_1 = U_1 S_1 V^\top_1,
\]
which is taken as an approximation to $E(t_1)$. As is shown in \cite{KLW16}, this is a first-order method that is robust to possibly small singular values of $E_0$ or $E_1$.

\subsection{Use of the low-rank structure in the eigenvalue computation}
For the computation of the gradient matrix $G_\eps(E)$, one needs to compute the  eigenvalue of smallest positive real part and the associated left and right eigenvectors of  the Hamiltonian matrix $M(X+L[\eps E])$. Except in the very first step of the algorithm, one can make use of the eigenvalue of smallest real part of the previous step in an inverse iteration (and possibly of the eigenvalues of second and third smallest real part etc.~to account for a possible exchange of the leading eigenvalue).

Moreover, for the choice \eqref{L-example} of $L$ we get from a perturbation $\Delta Z=\eps E$ with $E=U\Sigma V^\top$ of rank 8 that $C$ is perturbed by $\Delta C=\Delta Z\, Q^{-T}=\eps (U\Sigma) (Q^{-1}V)^\top$ of the same rank 8, which yields the perturbed Hamiltonian matrix
$$
M(X+L[\Delta Z])= M(X) + \Delta M, 
$$
where the perturbation 
\bcl $\Delta M$ is still of moderate rank in view of \eqref{M-p} and \eqref{M-b}.
\ecl
This fact can be put to good use in the computation of the required eigenvalues in the case of a high-dimensional system, using the Sherman--Morrison--Woodbury formula in an inverse iteration.

We further note that if $m,p\ll n$, then $M(X)$ can be viewed as a low-rank perturbation to the matrix
$$
\begin{pmatrix} A & 0 \\ 0 & -A^\top \end{pmatrix}.
$$
With the Sherman--Morrison--Woodbury formula, this can yield an efficient inverse iteration when $A$ is a large and sparse matrix for which shifted linear systems can be solved efficiently.


\section{Numerical illustrations}
\label{sec:num}

The goal of this section is to provide a few illustrative numerical examples. Both Problems 1 and 2 are considered 
\bcl for the bounded real case (passivity in the scattering representation / contractivity). \ecl  
All the numerical values are approximations of the computed values to few decimal digits.

\subsection{Example for Problem 1: passivity enforcement}

We consider an example from \cite{Grivet2004}. In this case the data is a linear control system whose Hamiltonian has some purely imaginary eigenvalues. Our solution strategy is to consider as initial point a perturbation of the system which makes it passive (i.e. which moves all the eigenvalues of the Hamiltonian matrix out of the imaginary axis) and then 
to optimize such a solution.  
The starting solution is computed by means of the algorithm in \cite{Grivet2004}. 
This algorithm perturbs only the matrix $C$  (so that $L$ has the form given by \eqref{L-example}).  

In the following we perform two different experiments.
\begin{itemize}
\item  We optimize in terms of the matrix $C$ only; 
\item  We run our optimization algorithm by perturbing the whole system;
\end{itemize}
 
The original system consists of the following matrices:
\begin{equation} \label{ex:2}
A=\begin{pmatrix}
-1/2 & 1 \\ -1 & -1/2
\end{pmatrix}, B = C^\top = \begin{pmatrix}
1/2 \\ 1/2
\end{pmatrix}, D = \begin{pmatrix}
1/2
\end{pmatrix}
\end{equation}
and the associated Hamiltonian matrix is
\[
\begin{pmatrix}
-\frac{1}{3} & \frac{7}{6} & \frac{1}{3} & \frac{1}{3} \\[2mm]
-\frac{5}{6} & -\frac{1}{3} &  \frac{1}{3} & \frac{1}{3} \\[2mm]
-\frac{1}{3} & -\frac{1}{3} & \frac{1}{3} & \frac{5}{6} \\[2mm]
-\frac{1}{3} & -\frac{1}{3}  & -\frac{7}{6} & \frac{1}{3}
\end{pmatrix}
\]
whose eigenvalues are two pairs of purely imaginary eigenvalues
\begin{equation}
\label{imageig}
\pm 1.1902\ldots \iu, \ \ \pm 0.8660\ldots \iu.
 \end{equation}
\subsubsection*{Perturbing the matrix $C$ only}

The first step is to get an initial approximation $X_0$ which moves all the eigenvalues off the imaginary axis. 
\begin{table}
	\begin{center}
		\begin{tabular}{|c|l|l|}
			it. num. & $f(\eps)$ & $\eps$                   \\
			\hline
			0   &   $0.435329$   &    $0.471928$   \\
			1   &   $0.424050$   &    $0.461975$    \\
			2   &   $0.340189$   &    $0.295810$   \\
			3   &   $0.219163$   &    $0.161991$     \\
			4   &   $0$          &    $0.081834$     \\
			5   &   $0.206323$   &    $0.152732$        \\
			6   &   $0.168366$   &    $0.131636$                        \\
			7   &   $0.110590$   &    $0.112236$          \\
			8   &   $0.005487$   &    $\textbf{0.10}1320$         \\
			9   &   $0.067009$   &    $\textbf{0.10}5804$       \\
			10  &   $0.036737$   &    $\textbf{0.10}2600$      \\
			11  &   $0.018486$   &    $\textbf{0.101}618$       \\
			12  &   $0.009431$   &    $\textbf{0.101}376$        \\
			13  &   $0.011275$   &    $\textbf{0.101}412$        \\
			14  &   $0.009672$   &    $\textbf{0.10138}0$              \\
			15  &   $0.009839$   &    $\textbf{0.10138}3$               \\
			16  &   $0.009957$   &    $\textbf{0.10138}6$                 \\
			17  &   $0.010044$   &    $\textbf{0.10138}7$               \\
			18  &   $0.009991$   &    $\textbf{0.101386}$            \\
			\hline
		\end{tabular}
	\caption{Convergence for passivity enforcement optimization - perturbation on C only: dynamic which moves an eigenvalue $\delta$-close to the imaginary axis \label{Tab2}}
\end{center}
\end{table}
By running the algorithm proposed in \cite{Grivet2004} we get (recalling that only $C$ is perturbed)
\[ 
C_0 = \left(0.2018 \qquad 0.4615 \right)
\]
and the eigenvalues of the Hamiltonian matrix $M(X_0)$ with $X_0=(A, B, C_0, D)$ are 
\[
 \pm 0.3199 \pm 1.0596 \iu.
\]
 
With our algorithms we are able to optimize such a starting solution by moving the eigenvalues closer to the imaginary axis, at a distance $\delta = 10^{-2}$.

First, with $L$ given by \eqref{L-example}, starting from the system  $X_0 = (A, B, C_0, D)$ we run our algorithm by perturbing only the matrix $C_0$ leaving $A, B, D$ unperturbed.  
Here we compute $Q$ such that $Q^\top\in \R^{n\times n}$ is a Choleski factor of the controllability Gramian $G_c=Q^\top Q$, as follows:
\begin{equation*} 
Q = \begin{pmatrix}
0.5916 & 0.0845 \\  0 &  0.3780
\end{pmatrix}
\end{equation*}
 
We get the solution 
\[
\widehat{C} = \left (0.3703 \qquad 0.6115 \right), 
\]
which is closer than $C_0$ to the data matrix $C$; in fact 
\[
\| (\widehat{C}  - C)Q^\top \|_F \approx 0.07941 < 0.18025 = \|(C_0 - C)Q^\top  \|_F.
\]
The eigenvalues of the Hamiltonian matrix $M(\widehat X)$ with $\widehat X=(A,B,\widehat C,D)$ 
are 
\[
\pm 0.9991 \cdot 10^{-2} \pm 1.0004 \iu 
\]
so that the $\delta$-closeness to the imaginary axis is achieved.  Similar convergence properties to the unconstrained case  can be observed in Table  \ref{Tab2}.

\subsubsection*{Perturbing the whole system}

In the second case examined, we  perturb the whole system $X_0 =(A, B, C_0, D)$ and we take $L={\rm Id}$ so that we minimize the Frobenius norm of the perturbation. The computed perturbed solution, corresponding to a  passive system, is given by $\widehat X =(\widehat A,\widehat  B, \widehat C, \widehat D)$ with
\[
 \begin{aligned}
 \widehat{A} &= \begin{pmatrix}
 -0.4592 & 1.0012 \\ -0.9968 & -0.4575
 \end{pmatrix}, \quad \widehat{B} = \begin{pmatrix}
 0.5169 \\ 0.5445
 \end{pmatrix} \\[3mm]
 \widehat{C} &= \begin{pmatrix}
 0.2477 & 0.5029
 \end{pmatrix}, \quad \widehat{D} = \begin{pmatrix}
 0.5587
 \end{pmatrix}
 \end{aligned}
\]
 and the associated Hamiltonian matrix $M(\widehat X)$
has eigenvalues
 $$
 \pm 0.9997 \cdot 10^{-2} \pm 1.0750 \iu.
 $$
The computed distance is
\[
\| \widehat{X} - X \|_F \approx 0.1141.
\]

\subsection{Example with Problem 2: distance to the nearest non-passive system}

This example is taken from \cite{Guglielmi2015}, where the Hamiltonian matrix nearness problem was solved without  
taking into account the original linear time-invariant system $X$. 
The data of the problem are the following matrices $A, B, C, D$:
\begin{equation} \label{ex:1}
A = \begin{pmatrix}
-8 & -4 & -1.5 \\ 4 & 0 & 0 \\ 0 & 1 & 0
\end{pmatrix}, B = \begin{pmatrix}
2 \\ 0 \\ 0 
\end{pmatrix}, C^\top = \begin{pmatrix}
1 \\ 1 \\ 0.75
\end{pmatrix}, D = \begin{pmatrix}
- 0.75
\end{pmatrix}
\end{equation}
hence the eigenvalues of the Haniltonian matrix $M(X)$ are
$$
\pm  6.5856, \ \ \pm 2.5784, \ \ \pm 0.5173.
$$
The positive eigenvalue which is closer to the imaginary axis is $\lambda = 0.5173\ldots$, which we aim 
to move $\delta$-close to the imaginary axis (with $\delta=10^{-2}$).
By running the proposed algorithm we get the following perturbed system $\widehat X =(\widehat A,\widehat  B, \widehat C, \widehat D)$:
\[
\begin{aligned}
\widehat{A} &= \begin{pmatrix}
-8.0008 & -4.0060 & -1.4577 \\ 3.9986 & -0.0102 & 0.0717 \\ -0.0023 & 0.9829 & 0.1196
\end{pmatrix}, \widehat{B} = \begin{pmatrix}
2.0142 \\ 0.0240 \\ 0.0399
\end{pmatrix}, \\  \widehat{C} &= \begin{pmatrix}
0.9991 & 0.9936  & 0.7978
\end{pmatrix}, \widehat{D} = \begin{pmatrix}
-0.7335
\end{pmatrix}.
\end{aligned}
\]
The corresponding eigenvalues of the Hamiltonian matrix $M(\widehat{X})$ are
$$
\pm 6.4472, \ \ \pm 2.7294, \ \ \pm 0.010.
$$
As expected the Hamiltonian matrix is different from the one computed in \cite{Guglielmi2015}, which does not correspond to any perturbed system.  Note that the
perturbed system is relatively close to the original one indicating a small distance to non-passivity.

In Table \ref{Tab1} we report the computed values in the outer iteration (i.e. $f(\eps):= \phi_\eps(\eps E(\eps))= \lambda(\eps)$), which alternates Newton steps to bisection steps.
\begin{table}
\begin{center}
\begin{tabular}{|c|l|l|l|}
 it. number	 & Newton/bisection  & $f(\eps)$  & value of $\eps$  \\
\hline
 0  &                & $0.517251$ & $0$ \\
 1  & initialization & $0.494635$ & $0.01$ \\
 2  &         Newton & $0$        & $0.254734$ \\ 
 3  &      bisection & $0.199354$ & $ 0.132367$ \\
 4  &         Newton & $0$        & $0.237252$ \\
 5  &      bisection & $0$        & $0.184809$ \\
 6  &      bisection & $0.077451$ & $0.158588$ \\
 7  &         Newton & $0$        & $0.171699$ \\
 8  &      bisection & $0$        & $0.\textbf{16}5143$ \\
 9  &      bisection & $0.043277$ & $0.\textbf{16}1866$ \\
 10 &         Newton & $0$        & $0.\textbf{163}505$ \\
 11 &      bisection & $0.029127$ & $0.\textbf{16}2685$ \\
 12 &         Newton & $0.018361$ & $0.\textbf{163}095$ \\
 13 &         Newton & $0.009044$ & $0.\textbf{163}300$ \\
 14 &      bisection & $0.014472$ & $0.\textbf{163}197$ \\
 15 &         Newton & $0.012067$ & $0.\textbf{1632}49$ \\
 16 &         Newton & $0.010662$ & $0.\textbf{1632}74$ \\
 17 &         Newton & $0.009886$ & $0.\textbf{163287}$ \\
\hline
\end{tabular}
\caption{Example \ref{ex:1}. Numerical values for $f(\eps)$  and $\eps$ for increasing number of (outer) iteration. \label{Tab1}}
\end{center}
\end{table}
We observe that the solution computed at iteration $13$ is good enough, though the algorithm needs few more iterations in order to achieve more accurate values 
in the computed solution of the optimization problem.

\subsection*{Acknowledgment}
The authors thank M.~Karow for useful suggestions and S.~Grivet-Talocia for  
providing interesting examples for Problem 2.

\bcl We are grateful to two anonymous referees and to S.~Grivet-Talocia for their constructive comments on a previous version of this paper. \ecl

Nicola Guglielmi acknowledges that his research was supported by funds from the
Italian MUR (Ministero dell'Universit\`a e della Ricerca) within the PRIN 2017
Project ``Discontinuous dynamical systems: theory, numerics and applications''
and by the INdAM Research group GNCS (Gruppo Nazionale di Calcolo Scientifico).

\bigskip

\bibliographystyle{plain}
\bibliography{hamiltonian}

\begin{thebibliography}{10}

\bibitem{Alam2011}
R.~Alam, S.~Bora, M.~Karow, V.~Mehrmann, and J.~Moro.
\newblock Perturbation theory for {H}amiltonian matrices and the distance to
  bounded-realness.
\newblock {\em SIAM J. Matrix Anal. Appl.}, 32(2):484--514, 2011.

\bibitem{BenLMV15}
P.~Benner, P.~Losse, V.~Mehrmann, and M.~Voigt.
\newblock Numerical linear algebra methods for linear differential-algebraic
  equations.
\newblock In {\em Surveys in differential-algebraic equations III}, pages
  117--175. Springer, 2015.

\bibitem{slicot}
P.~Benner, V.~Mehrmann, V.~Sima, S.~Van~Huffel, and A.~Varga.
\newblock Slicot --- a subroutine library in systems and control theory.
\newblock In {\em Applied and computational control, signals, and circuits},
  pages 499--539. Springer, 1999.

\bibitem{Boyd1989}
S.~Boyd, V.~Balakrishnan, and P.~Kabamba.
\newblock A bisection method for computing the ${H}_{\infty}$ norm of a
  transfer matrix and related problems.
\newblock {\em Math. Control, Signals and Systems}, 2:207--219, 1989.

\bibitem{BruS13}
T.~Br{\"u}ll and C.~Schr{\"o}der.
\newblock Dissipativity enforcement via perturbation of para-{H}ermitian
  pencils.
\newblock {\em IEEE Trans. Circuits and Systems I: Regular Papers},
  60(1):164--177, 2012.

\bibitem{GilS17}
N.~Gillis and P.~Sharma.
\newblock On computing the distance to stability for matrices using linear
  dissipative {H}amiltonian systems.
\newblock {\em Automatica}, 85:113--121, 2017.

\bibitem{GilS18}
N.~Gillis and P.~Sharma.
\newblock Finding the nearest positive-real system.
\newblock {\em SIAM J. Numer. Anal.}, 56(2):1022--1047, 2018.

\bibitem{Grivet2004}
S.~Grivet-Talocia.
\newblock Passivity enforcement via perturbation of {H}amiltonian matrices.
\newblock {\em IEEE Trans. Circuits and Systems I: Regular Papers},
  51(9):1755--1769, 2004.

\bibitem{Grivet2015}
S.~Grivet-Talocia and B.~Gustavsen.
\newblock {\em Passive macromodeling : theory and applications}.
\newblock Wiley Series in Microwave and Optical Engineering. Wiley, Hoboken,
  New Jersey, 2015.

\bibitem{Guglielmi2015}
N.~Guglielmi, D.~Kressner, and C.~Lubich.
\newblock Low rank differential equations for {H}amiltonian matrix nearness
  problems.
\newblock {\em Numer. Math.}, 129(2):279--319, 2015.

\bibitem{Guglielmi2017}
N.~Guglielmi and C.~Lubich.
\newblock Matrix stabilization using differential equations.
\newblock {\em SIAM J. Numer. Anal.}, 55(6):3097--3119, 2017.

\bibitem{GO11}
N.~Guglielmi and M.~Overton.
\newblock Fast algorithms for the approximation of the pseudospectral abscissa
  and pseudospectral radius of a matrix.
\newblock {\em SIAM J. Matrix Anal. Appl.}, 32:1166--1192, 2011.

\bibitem{Kato1976}
T.~Kato.
\newblock {\em {Perturbation theory for linear operators; 2nd ed.}}
\newblock Grundlehren Math. Wiss. Springer, Berlin, 1976.

\bibitem{KLW16}
E.~Kieri, C.~Lubich, and H.~Walach.
\newblock Discretized dynamical low-rank approximation in the presence of small
  singular values.
\newblock {\em SIAM J. Numer. Anal.}, 54:1020--1038, 2016.

\bibitem{Koch2007}
O.~Koch and C.~Lubich.
\newblock Dynamical low-rank approximation.
\newblock {\em SIAM J. Matrix Anal. Appl.}, 29(2):434--454, 2007.

\bibitem{KotBBP19}
A.~Kothyari, C.~Bhawal, M.~N. Belur, and D.~Pal.
\newblock Defective {H}amiltonian matrix imaginary eigenvalues and
  losslessness.
\newblock In {\em 2019 Fifth Indian Control Conference (ICC)}, pages 166--171.
  IEEE, 2019.

\bibitem{LO14}
C.~Lubich and I.~V. Oseledets.
\newblock A projector-splitting integrator for dynamical low-rank
  approximation.
\newblock {\em BIT}, 54:171--188, 2014.

\bibitem{MehMS16}
C.~Mehl, V.~Mehrmann, and P.~Sharma.
\newblock Stability radii for linear {H}amiltonian systems with dissipation
  under structure-preserving perturbations.
\newblock {\em SIAM J. Matrix Anal. Appl.}, 37(4):1625--1654, 2016.

\bibitem{MehMS17}
C.~Mehl, V.~Mehrmann, and P.~Sharma.
\newblock Stability radii for real linear {H}amiltonian systems with perturbed
  dissipation.
\newblock {\em BIT}, 57(3):811--843, 2017.

\bibitem{MS88}
C.D. Meyer and G.W. Stewart.
\newblock Derivatives and perturbations of eigenvectors.
\newblock {\em SIAM J. Numer. Anal.}, 25:679--691, 1988.

\bibitem{OveVD05}
M.~L. Overton and P.~Van~Dooren.
\newblock On computing the complex passivity radius.
\newblock In {\em Proceedings of the 44th IEEE Conference on Decision and
  Control}, pages 7960--7964. IEEE, 2005.

\bibitem{Paige81}
C.~Paige and C.~Van~Loan.
\newblock A {S}chur decomposition for {H}amiltonian matrices.
\newblock {\em Linear Algebra Appl.}, 41:11--32, 1981.

\bibitem{SchS07}
C.~Schr{\"o}der and T.~Stykel.
\newblock Passivation of {LTI} systems.
\newblock {\em Preprint 368, MATHEON, Berlin}, 2007.

\end{thebibliography}

\end{document}